\documentclass{article}
\usepackage[utf8]{inputenc}
\usepackage{amsmath}
\usepackage{amsfonts}
\usepackage{dsfont}
\usepackage{amsthm}
\usepackage{color}
\usepackage{xcolor}
\usepackage{hyperref}

\newtheorem{thm}{Theorem}[section]
\newtheorem{lem}{Lemma}[section]
\newtheorem{prop}{Proposition}[section]
\newtheorem{cor}{Corollary}[section]
\newtheorem{rem}{Remark}[section]

\title{Exponential Asymptotics of Ricci Flow Neckpinch}
\author{Hamidreza Mahmoudian}
\date{}

\begin{document}

\maketitle 
\begin{center}
    \large Abstract
\end{center}

In this article we examine the formation of cylindrical neckpinch singularities in Ricci flow in compact manifolds. Rigorous examples of neckpinch singularity for rotationally symmetric initial data were first constructed by Angenent and Knopf in \cite{neckpinch}, and examples with given asymptotic behaviour were constructed in \cite{precise} and \cite{degenneck}. Here we discuss the asymptotic behaviour of the flow under Type-I assumption for general symmetric initial data, and show the asymptotic profiles in \cite{precise} and \cite{degenneck} are the only possibilities.

\section{Introduction}
Let $M^{n+1}$ be a closed Riemannian manifold and let $g(t)$ be a solution of Ricci flow
$$\frac{\partial g}{\partial t} = -2 Rc[g(t)] \qquad 0\leq t < T$$
on $M$ with first singularity time $T$. By Hamilton's classical result \cite{hamilton} this is equivalent to
$$\limsup_{t \to T} \sup_M |Rm| = \infty$$

For any such Ricci flow, we can define its \textit{dynamic blow-up} by $$\tilde{g}(\tau) = \frac{g(t)}{T-t} \quad \text{where} \quad \tau=-\log(T-t)$$
More generally, we can define a \textit{sequential blow up} to be any pointed sequence of the form $(M, g_i(t), p_i)$, where $p_i \in M$ and
$$g_i(t) = \lambda_i g(t_i+ \frac{t}{\lambda_i})$$
for some $t_i \to T$ and $\lambda_i \to \infty$.
\newpage

We say the flow \textit{develops a neckpinch singularity at $(p,T)$} if $$(M, \tilde{g}(\tau), p) \to (\mathbb{S}^{n} \times \mathbb{R}, 2(n-1)g_{can}\oplus g_{euc}, p_\infty)$$ in the smooth Gromov-Hausdorff sense. Here $g_{can}$ and $g_{euc}$ are the canonical round metric of the unit sphere $\mathbb{S}^{n}$ and the Euclidean metric of $\mathbb{R}$. In this work, we examine the specific rate of convergence to cylinder and the asymptotic shape of $\tilde{g}$ for a neckpinch under the assumption of rotational symmetry.

In recent years, there has been much progress regarding both analysis of limit spaces arising after blow ups of Ricci flow solutions, and applications of fine asymptotic analysis of such convergence. 

Recall that the flow is said to be \textit{(globally) Type-I} if it satisfies
$$\sup_M |Rm(p, t)| \leq \frac{C}{T-t}$$
Limits of Type-I blowups around a sequence $(p,t_i)$ of space-time points, with a fixed $p \in M$, have been studied by Naber \cite{Nabershrinker}, Enders-Muller-Topping \cite{EMTshrinker} and Muller-Mantegazza \cite{shrinkerlimit}. Under Type-I assumption, such blowup sequences will converge to a gradient shrinker, which will be non-flat if $p$ is a singular point (c.f. \cite{EMTshrinker}). For sequences $(p_i, t_i)$ with varying base point in $M$, Bamler's recent theory of \textit{F-convergence} shows convergence to (possibly singular) Ricci solitons in a weak sense. It is worth pointing out that Ricci shrinkers with vanishing Weyl tensor have been classified by combined work of \cite{rotsymshrink} and \cite{harmweyl}: The only examples are the round spheres $\mathbb{S}^n$, the round cylinders $\mathbb{S}^n\times \mathbb{R}$, and the flat space $\mathbb{R}^n$.

On the other hand, when discussing asymptotic shrinkers arising from rescaling ancient solutions, a fine description of the convergence has been proven useful in classification of the ancient solutions. For instance, under the assumptions of rotational symmetry, such asymptotic methods have been used to prove uniqueness of Bryant soliton among certain classes of complete $\kappa-$solutions; in dimension $3$ among solutions with nonnegative curvature and positive sectional curvature by Brendle \cite{brendlebarrier}, and in dimension $n\geq4$ among uniformly PIC and weakly PIC2 solutions by Li and Zhang \cite{LiZhang} (See also \cite{BreKnaff} for removing the symmetry assumption). Similarly, for \textit{compact} ancient Ricci flows, the precise asymptotic behaviour derived in \cite{compactancient3d} has been used in \cite{BDSK}to show uniqueness of Perelman solutions as the only possibility besides round spheres. Similar classifications have been done for mean curvature flow.

In the context of neckpinch singularities, the first rigorous examples of the phenomenon were constructed by Angenent and Knopf in \cite{neckpinch}. These are Type-I examples starting from a dumbbell-shaped initial data, with two large spheres joined by a thin cylinder neck. For an open set of such initial data, precise asymptotic behavior was derived in \cite{precise}. As one might expect, these examples become singular and \textit{pinch off} in a single point. On the other hand, examples of \textit{degenerate} neckpinches were constructed by Angenent, Isenberg and Knopf in \cite{degenneck}. By degenerate, we mean the rescaled flow converges to a cylinder or the Bryant soliton, depending on the choice of blow up sequence. Contrary to the previous case, these examples are Type-II and become singular on an entire compact subset of the manifold.

In this work, it is shown that the behaviors constructed in those examples are the only possibilities in rotationally symmetric Ricci flows. One motivation for deriving fine asymptotic information is proving that singularities are generically isolated points. More generally, knowledge of asymptotic behavior on large scales combined with knowledge of possible singularity models will lead to information about the structure of high curvature regions. In our situation, it would be desirable to exclude \textit{long tubes} that do not close with a cap. These would correspond to a compact subset somewhere in the middle of the manifold (as opposed to a region containing the umbilical \textit{tips} on either end) that is becoming singular. 

\textbf{Acknowledgements}. I would like to thank Natasa Sesum, Dan Knopf, Sigurd Angenent and Max Hallgren for their numerous helpful comments and suggestions during the development of this work. 
\section{Main Result and Outline of The Proof}

\subsection{Statement of The Main Theorem}

In order to state our result precisely, we follow the terminology in \cite{degenneck, neckpinch, precise}.

We assume $M$ is diffeomorphic to $S^{n+1}$ which we regard as $[-1,1]\times S^n$ with $\{-1\} \times S^n$ and $\{1\} \times S^n$ identified to single points $S$ and $N$. We say $g(t)$ is \textit{rotationally symmetric}, if it is preserved by the $O(n+1)$ action on $S^n$ fibres. In that case we can write
$$g(t) = \varphi^2(t,x) dx^2 + \psi^2(t,x) g_{can}$$
We call $\psi$ the \textit{radius}, and the $S^n$-orbit $x=0$ \textit{the equator}. We say $g(t)$ has \textit{reflection symmetry} if $\varphi$ and $\psi$ can be chosen to be even functions of $x$ (Both rotation and reflection symmetry are preserved by the flow).

It is convenient to work with the geometric coordinate $s$, defined as
\begin{gather*}
    s(t,x) = \int_0^x \varphi(t,r) \, dr
\end{gather*}
It measures the geodesic distance between the $O(n+1)$-orbits at the equator and at $x$. We call a local minimum of $\psi$ a \textit{neck} and a local maximum of $\psi$ a \textit{bump}. We also refer to the points at $x=\pm1$ as \emph{Poles} or \emph{Tips} and to the regions between a pole and the first bump next to it as a \emph{Polar Caps}.

Our main result is the following Theorem:

\begin{thm} \label{main}
Let $(M, g(t)), \, 0 \leq t <T$ be a Type-I solution to Ricci flow with first singular time $T$ that has rotational and reflection symmetry as above.

Assume $\lim_{t\to T} \psi(0,t)=0$, and that $g(t)$ develops a neckpinch singularity at the equator as $t\to T$. In particular,
$$u(\tau, \sigma) = \frac{\psi(t, s)}{\sqrt{2(n-1)(T-t)}} \to 1 \qquad \text{uniformly for}\,\, |\sigma| \leq R \,\, \text{as} \,\, \tau \to \infty$$
for $\sigma=\frac{s}{\sqrt{T-t}}$ and $\tau=-\log(T-t)$. Then one of the following holds:
\begin{enumerate}
    \item $|u-1| \leq C_A \exp(-A\tau)$ for every $A$, for $\tau \geq \tau_A$ and $|\sigma| \leq A\sqrt{\tau}$.
    \item We have
    $$u(\tau, \sigma) = 1 + \frac{\sigma^2-2}{8\tau} + o(\frac{1}{\tau})$$
    uniformly over compact sets $|\sigma| \leq R$.
    \item For some positive integer $m \geq 3$ and some constant $C$, we have
    $$u(\tau, \sigma) = 1 + C e^{\frac{2-m}{2} \tau} h_m(\sigma) + o(e^{\frac{2-m}{2} \tau})$$
    uniformly over compact sets $|\sigma| \leq R$. Here $h_m$ is the $m$-th modified Hermite polynomial.
\end{enumerate}
\end{thm}

\subsection{Remarks About Assumptions}

Some remarks about the assumptions in our main Theorem are in order. The reflection symmetry assumption is for controlling the location of the singularity. Without reflection symmetry, we have two natural choices: We can rescale at a neck and choose $\sigma=0$ to correspond to the time dependent location of a minimum $\xi(t)$, or we can rescale at a fixed point $p$. The problem with both these choices is the lack of information about how the necks are moving. The first choice gives us terms coming from $\xi'(t)$ that we cannot control. With the second choice, the region with good asymptotic might be too far from the origin once we rescale. Considering the examples in \cite{degenneck}, this will happen if we pick the wrong point.

The Type-I assumption plays two roles in our analysis. On one hand, it is important  for convergence to the cylinder. However, it is reasonable to expect that a local Type-I bound should suffice for establishing the convergence. On the other hand, once we assume the convergence, our work does not use the Type-I condition in any essential way. What we actually need is a lower diameter bound on the part of $M$ that is away from the tips.

Considering the results in \cite{shrinkerlimit} and Kotschwar's classification in \cite{classkot}, assuming cylindrical neckpinches is not too restrictive. Note, however, that assuming the common definition of non-degenerate neckpinch singularity, which says the smooth Gromov-Hausdorff limit of parabolically rescaled metric around every sequence of space-time points is a cylinder, does not automatically imply $u \to 1$. This is because the functions $\psi$ and $u$ are not geometric quantities. Ideally this should be proven only using the PDE structure, similar to classic blow up results for heat equation. We expect the results of \cite{rotsym} to be useful.

The first case in our theorem is unlikely considering the non-degeneracy result of Lax \cite{laxnondegen}. Similar situations have been dealt with by Sesum \cite{expdecayse} and Kotschwar \cite{expdecaykot}. In \cite{expdecaykot}, the argument depends heavily on Lojaseiwicz inequality satisfied by Perelman's $\mu$ functional near compact solitons, which was proven in \cite{lojas} by Sun and Wang. Such an inequality is shown for cylinders in mean curvature flow by Colding and Minicozzi, but it is not yet available for Ricci flow.

\subsection{Outline of The Proof}\label{outline}

We now discuss the precise setup of the problem and explain the main outline of the proof. We follow the conventions and notations in \cite{precise}, and review some of the computations presented in \cite{neckpinch} and \cite{precise} without proof.

If $g(t)$ is a rotationally symmetric Ricci flow, the radius function $\psi$ and the function $\varphi$ satisfy:
\begin{gather}
    \psi_t = \psi_{ss}-(n-1)\frac{1-\psi_s^2}{\psi} \\
    \varphi_t = n \frac{\psi_{ss}}{\psi} \varphi \nonumber
\end{gather}
These can be derived using the equation for $g(t)$ once we note
\begin{gather*}
    Ric = -n\frac{\psi_{ss}}{\psi} \, ds^2 + \big( -\psi \psi_{ss} -(n-1)\psi_s^2 + (n-1)\big) g_{can}
\end{gather*}
Note that the variables $s$ and $t$ do not commute. Their commutator is given by
\begin{gather*}
    [\frac{\partial}{\partial t} , \frac{\partial}{\partial s}] = -n\frac{\psi_{ss}}{\psi} \frac{\partial}{\partial s}
\end{gather*}

We define the blown-up radius and the self-similar variables as 
\begin{gather}
    u(\tau, \sigma)=\frac{\psi(t,s)}{\sqrt{2(n-1)(T-t)}}\\
    \sigma=\frac{s}{\sqrt{T-t}} \nonumber \\
    \tau= - \log(T-t) \nonumber
\end{gather}
These variables satisfy
\begin{gather*}
    [\frac{\partial}{\partial \tau} , \frac{\partial}{\partial \sigma}] = -\left( \frac{1}{2} + n\frac{u_{\sigma\sigma}}{u} \right) \frac{\partial}{\partial \sigma}
\end{gather*}
The blown up radius satisfies the following equation in non-commuting variables
\begin{gather}
u_\tau=u_{\sigma \sigma} + \frac{1}{2}\left( u-\frac{1}{u}\right) + (n-1) \frac{u_\sigma^2}{u}
\end{gather}
We can compute the equation for $u$ in commuting variables by adding an integral term $J$ .
\begin{lem}\label{ueqcom}
Assuming reflection symmetry for all $t$ close enough to $T$, the equation for $u$ in commuting variables is
\begin{gather}
    u_\tau=u_{\sigma \sigma} -\frac{\sigma}{2}u_\sigma - nJ(\tau, \sigma) u_\sigma + \frac{1}{2}\left( u-\frac{1}{u}\right) + (n-1) \frac{u_\sigma^2}{u}\\
    J(\tau, \sigma)= \int_0^\sigma \frac{ u_{\sigma \sigma}}{u}\,d\sigma = \frac{u_\sigma}{u}(\tau, \sigma) +\int_0^\sigma \left(\frac{u_\sigma}{u} \right)^2\,d\sigma
\end{gather}
\end{lem}

We will work with the quantities
\begin{gather}
    U = \log u \qquad f= U_\sigma = \frac{u_\sigma}{u}
\end{gather}
The following evolution equations are straightforward to derive.
\begin{lem} \label{ufeq}
    The quantities $U=\log u$ and $f=\frac{u_\sigma}{u}$ satisfy
    \begin{gather}
    U_\tau= \mathcal{L}U + \frac{1}{2}\left( 1-e^{-2U}-2U\right) - n\left( \int_0^\sigma f^2 \right) f\\
    f_\tau= \mathcal{A}f + \left( \frac{1}{u^2}-1 \right)f -nf^3 -n\left( \int_0^\sigma f^2  \right)f_\sigma 
    \end{gather}
    where the operators $\mathcal{L}$ and $\mathcal{A}$ are defined as
\begin{gather}
    \mathcal{L} = \frac{\partial^2}{\partial \sigma^2} - \frac{\sigma}{2} \frac{\partial}{\partial \sigma} + 1\\
    \mathcal{A} = \mathcal{L} - \frac{1}{2}=\frac{\partial^2}{\partial \sigma^2} - \frac{\sigma}{2} \frac{\partial}{\partial \sigma} + \frac{1}{2}
\end{gather}
\end{lem}
\begin{proof}
    The equation for $U$ is easy to derive. For the equation for $f=U_\sigma$, we can differentiate the equation for $U$ to get (not that there are no commutator terms)
    \begin{align*}
        & f_\tau = \mathcal{L}U_\sigma - \frac{1}{2}U_\sigma + \frac{1}{2}\left( 2e^{-2U}-2\right)U_\sigma - nf^3 -n\left( \int_0^\sigma f^2 \right) f_\sigma \\
        & = (\mathcal{L} - \frac{1}{2})f + \left( \frac{1}{u^2}-1 \right)f - nf^3 -n\left( \int_0^\sigma f^2 \right) f_\sigma
    \end{align*}
\end{proof}
When $u$ is $C^2$-close to $1$, we have $U\approx u-1$ and we expect the non-linearities to be small of quadratic and cubic orders.
In asymptotic analysis of ancient solutions such as \cite{brendlebarrier} or \cite{compactancient3d}, barriers related to Bryant soliton play an important role. These barriers are singular near the tip and have controlled boundary conditions near the center and near the tip. In our case, we have candidate barriers from \cite{degenneck}. However, we are dealing with a forward equation and controlling the boundary conditions for the solution needs somewhat sharp estimates which we do not have. We therefore choose to follow ideas of Filippas and Kohn from \cite{filipkohn} instead, which we now briefly describe.

Consider the equation
\begin{gather*}
    q_\tau = \mathcal{L}q + q^2 \qquad q \to 0 \,\,\text{on compact sets}
\end{gather*}
The operator $\mathcal{L}$ is self-adjoint on the weighted Hilbert space $$\mathfrak{H}=L^2(\rho\, d\sigma) \qquad \rho(\sigma) = \exp(\frac{-\sigma^2}{4})$$
with eigenvalues $$-\mu_m=1-\frac{m}{2} \qquad m=0,1,2, \dots$$
Since $q \to 0$, we expect the projections of $q$ on neutral and negative eigenspaces to dominate the projections on positive eigenspaces. Define $x(\tau), y(\tau), z(\tau)$ to be the norms of projections of $q$ on unstable, neutral and stable eigenspaces of $\mathcal{L}$ respectively. These quantities satisfy
\begin{gather*}
    \frac{dx}{d\tau}  \geq \frac{1}{2}x - \| q^2 \| \\
    |\frac{dy}{d\tau} | \leq  \| q^2 \| \\
    \frac{dz}{d\tau}  \leq -\frac{1}{2}z + \| q^2 \|
\end{gather*}

In order to control $\|q^2\|$, note that if $q\to 0$ on compact sets, for any $k$ and any $\varepsilon$ and $\delta$ we eventually have
$$
\int q^4 \rho = \int_{|\sigma|\leq \delta^{-1}} q^4 \rho+ \int_{|\sigma|> \delta^{-1}} q^4 \rho \leq \varepsilon^2 \int q^2 \rho + \delta^k \int q^4 \sigma^k \rho
$$
It can be shown that for large $k$, for any $\varepsilon$ the quantity $I^2= \int q^4 \sigma^k \rho$ eventually satisfies the inequality
$$
\frac{d}{d\tau} I \leq -\frac{1}{2} I + \varepsilon \| q \|
$$

Therefore by replacing $z$ with $\zeta= z+ I$ we get the system of inequalities
\begin{gather}\label{odesys}
    \frac{dx}{d\tau}  \geq \frac{1}{2}x - \varepsilon(x+y+\zeta)  \nonumber \\
    |\frac{dy}{d\tau} | \leq  \varepsilon(x+y+\zeta) \\
    \frac{d \zeta}{d\tau}  \leq -\frac{1}{2}\zeta + \varepsilon(x+y+\zeta) \nonumber
\end{gather}
and we can apply the ODE lemma of Merle and Zaag \cite{merlezaag} to conclude $x+\zeta = o(y)$ or $x+y=o(\zeta)$.

The strategy described above is not directly applicable to $u-1$ or $U$ for multiple reasons. First, the non-linear terms are non-local and depend on the first order term $f=\frac{u_\sigma}{u}$. For this reason, we choose to work with $f$
instead. This basically allows us to control the non-local term with only weak $C^2$ control of $u$. However, this comes at the cost of losing information about the first mode of $U \approx u-1$ corresponding to eigenvalue $1$ (this can be thought of as a constant which is killed with differentiation). We therefore need to control this mode separately.

Another major complication is the error terms coming from localizing the quantities with a cut-off function. We expect the convergence to cylinder to be slower than $\exp(-\Lambda \tau)$ for some $\Lambda \gg 1$. Therefore, due the presence of Gaussian weight, it is natural to take a cut-off with a support of the size $A\sqrt{\tau}$. To deal with such error terms, we use a technical modification of the well known Merle-Zaag ODE lemma, Lemma \ref{myMZ}. Among other things, we show that the lower bound $\exp(-\Lambda \tau)$ on a sequence of times implies the same bound for all times. Therefore, by choosing $A$ in the cut-off large compared to $\Lambda$, we can guarantee negligible errors. However, for all of this to make sense, we need to show that the diameter of the rescaled metric is larger than $A\sqrt{\tau}$ for any $A$. This is done using the Type-I condition, essentially by showing the existence of some non-singular region on either side of the cylindrical part. 

In section \ref{apriori}, we derive some apriori estimates for $u$. First, we use the Type-I condition to show that the two tips cannot become singular. This will show that the tip regions move away exponentially fast. In particular, the distance between the neckpinch and the tip regions will be greater than $A\sqrt{\tau}$ for all $A$, and we have $u \geq \frac{1}{2}$ on $|\sigma|=O(\sqrt{\tau})$. Then, we combine this lower bound, the Hamilton-Ivey pinching estimate proven by Zhang in \cite{hamivy} and a barrier argument to derive $C^2$ bounds. As discussed below, Sturmian theory for parabolic equations allows us to assume a fixed number of necks and bumps, and we essentially have singularities only at the points with $\psi \to 0$. The barrier argument is good for the region between a neck and a bump that are both becoming singular, and the Hamilton-Ivey estimate is useful for the region between the last singular neck and the bump next to it.

In section \ref{asympt}, we derive the necessary estimates to imitate the Filippas-Kohn approach for $f$. We then state and apply our modified ODE lemma to prove that exactly one of the following scenarios happen: The neutral eigenfunction of $\mathcal{L}$ controls $f$, or $\| f\|$ converges to zero exponentially, or $\| f \|$ decays faster than any exponential. We also show how to control the remaining mode of $u-1$ using $f$, and present the proof of our main Theorem.

In sections \ref{neutral} and \ref{exp}, we deal with the neutral and exponential cases and derive the ODE for the coefficient of the dominant mode. This will allow us to prove the asymptotic behavior of $u-1$ on compact sets described in the remaining two cases of Theorem \ref{main}.

\section{Apriori Estimates}\label{apriori}

The goal of this section is to prove the following apriori bounds: For $|\sigma|\leq 4A \sqrt{\tau}$ and $\tau \geq \tau_0$ we have 
\begin{gather}
    1-\frac{C_0}{\tau} \leq u \leq C_A\\
    |u_\sigma| \leq \frac{C_A}{\sqrt{\tau}} 
\end{gather}
Note that by the asymptotic analysis in \cite{precise}, these bounds are sharp.

\subsection{Analysis of Polar Caps}

We first show that the poles cannot be singular points themselves.

\begin{prop}
Let $p$ denote one of the poles of $M$. For any sequence of times $t_i \to T$, let $M_i=M$, $g_i(t) = \frac{1}{T-t_i} g(t_i + t(T-t_i))$. Then the pointed sequence $(M_i, g_i(0), p)$ has a subsequence converging to the flat space.
\end{prop}
\begin{proof}
Since $T-t_i \to 0$, a subsequence (still indexed by $i$) converges to a gradient shrinking soliton $(M_\infty, g_\infty, p_\infty)$ by \cite{shrinkerlimit}. Since $(M, g(t))$ has vanishing Weyl tensor, so does $M_\infty$. The results of \cite{harmweyl} and \cite{rotsymshrink} show that $M_\infty$ is either the sphere $\mathbb{S}^{n+1}$, the cylinder $\mathbb{S}^n \times \mathbb{R}$, or the flat space $\mathbb{R}^{n+1}$.

\textit{The Cylinder Case}: This is the easiest case to exclude. We simply note that the point $p$ is an umbilical point for $(M, g_i(0))$, where as the cylinder $\mathbb{S}^n \times \mathbb{R}$ has no umbilical points.

\textit{The Sphere Case}: The idea is to note that on the geodesics of $M_i$ starting from $p$ we will eventually see the cylinder geometry, whereas on a geodesic of the sphere we always see positive curvatures. There is an exhaustion $\mathcal{U}_i$ of $\mathbb{S}^{n+1}$ and diffeomorphisms $\phi_i:\mathcal{U}_i \to M_i$ such that $\phi_i(p_\infty)=p$ and $\| g_\infty - \phi_i^* g_i(0)\|_{g_\infty} \to 0$.

Let $g_i$ denote $\phi_i^* g_i(0)$ for short, and let us identify $p$ and $p_\infty$ and work on the sphere. Consider the $g_i$-geodesic $\gamma_i$ starting at $p$ with $g_i$-unit tangent vector $V_i$. Since $g_i \to g_\infty$ on $T_p \mathbb{S}^{n+1}$, there is a subsequence of $V_i$ that converges to a $g_\infty$-unit vector $V$ with respect to $g_\infty$. Let $\gamma$ be the $g_\infty$-geodesic starting at $p$ and tangent to $V$ (so it is a great circle). We may take a neighborhood $\mathcal{U}$ of $\gamma$ and assume $\mathcal{U} \subseteq \mathcal{U}_i$.

Parametrize $\gamma_i$ by the $g_i$-length and $\gamma$ by $g_\infty$-length. For any fixed $D$, $\gamma_i(r)$ uniformly converges to $\gamma(r)$ with respect to $g_\infty$ for $0 \leq r \leq D$. Since $\gamma$ will pass through $p$ again for the first time after a fixed length $D$, we conclude $\gamma_i(D)$ converges to $p$ with respect to $g_\infty$ and therefore with respect to $g_i$. We may also assume that $\gamma_i(r)$ lies in $\mathcal{U}$ for $0\leq r\leq D$ and that there is a point $q_i$ on $\gamma_i$ which is at least $D/4$ away with respect to $g_i$.

Now if we pull $\gamma_i$ back to $M_i$ with $\phi_i$ (which is possible sine $\gamma_i \subset \mathcal{U}_i$), we see that $\gamma_i$ is a geodesic on $M_i$ that starts at $p$, has points at least $D/4$ distance away from $p$, and is such that $\gamma_i(D)$ is arbitrary close to $p$. Since $M_i$ is rotationally symmetric with a pole at $p$, the only way this is possible is for $\gamma_i$ to pass the other pole as well. But this is not possible if there is a neck-like region between the two poles.
\end{proof}

Now we show that an almost flat rotationally symmetric manifold must have an almost linear profile function.

\begin{lem}
Let $p$ be a pole of $M$ and let $s_1(x,t)$ denote the $g(t)$-distance of $x$ from the pole $p$. Consider a sequence of times $t_i\to T$ and define $g_i(t) = \frac{1}{T-t_i} g(t_i + t(T-t_i))$. If the sequence $(M, g_i(0), p)$ converges to the flat space $(\mathbb{R}^{n+1}, \delta, 0)$, the rescaled radii $u(\tau_i, \sigma_1)$ converge to $\frac{\sigma_1}{\sqrt{2(n-1)}}$ uniformly on compact sets as $i \to \infty$. Here $\sigma_1=\frac{s_1}{\sqrt{T-t}}$.
\end{lem}
\begin{proof}
Fix a number $R$ and consider the sectional curvature of $(M, g_i(0))$ equal to $\frac{1-2(n-1)u_\sigma^2}{u^2}$. Since we have convergence to flat space, this quantity goes to zero uniformly on fixed $g_i$ distances from the pole $p$. If for some small $\varepsilon$ we have $|\frac{1-2(n-1)u_\sigma^2}{u^2}| \leq \varepsilon$, then as long as $u \leq 2R$ we get  $\frac{1-4\varepsilon R^2}{2(n-1)} \leq u_\sigma^2 \leq \frac{1+4\varepsilon R^2}{2(n-1)}$. This implies $$\sqrt{1-4\varepsilon R^2}\frac{\sigma_1}{\sqrt{2(n-1)}} \leq u \leq \sqrt{1+4\varepsilon R^2}\frac{\sigma_1}{\sqrt{2(n-1)}}$$
By taking $\varepsilon$ small depending on $R$, the right hand side will be strictly less than $2R$ if $\sigma_1 \leq R$ and we are done.
\end{proof}

\subsection{Estimates For $u$ and $u_{\sigma\sigma}$}

We first recall two simple lemmas proven in \cite{neckpinch}. Following \cite{neckpinch}, we consider the derivative of the radius and the scale invariant difference of sectional curvatures
\begin{gather*}
v=\psi_s\\
a= \psi \psi_{ss} - \psi_s^2+1
\end{gather*}
These quantities satisfy the equations
\begin{gather*}
    v_t= v_{ss}+ \frac{n-2}{\psi} v v_s + (n-1)\frac{(1-v^2)}{\psi^2} v\\
    a_t = a_{ss}+ (n-4) \frac{\psi_s}{\psi} a_s - 4(n-1)\frac{\psi_s^2}{\psi^2}a
\end{gather*}
\begin{lem}
We have the bounds
\begin{gather*}
    \sup |v(t,s)| \leq \max \{ 1, \sup |v(0,s)| \}\\
    \sup |a(t,s)| \leq \sup |a(0,s)|\\
    |\psi \psi_{ss}|\leq C\\
    |2\psi \psi_t| \leq C\\
    |Rm| \leq \frac{C}{\psi^2}
\end{gather*}
\end{lem}
\begin{proof}
Apply the maximum principle to equations for $v$ and $a$ for the first two bounds. They imply the third. The other two follow from Proposition 5.3 and Lemma 7.1 of \cite{neckpinch} since only the first two bounds are used in their proofs.
\end{proof}

According to Sturmian theory, for $0<t$ the number of roots of $v=\psi_s$ is finite and non-increasing. Therefore we may assume the number of necks and bumps is constant and their location defines well-defined non-intersecting curves. Our assumptions imply $v(t,0)=0$ and $\psi(t,0) \to 0$ as $t\to T$, so the rough shape of the singularity around $s=0$ is a number of consecutive necks and bumps (and the region between them and possibly the polar caps) having $\psi \to 0$. We can say a little more.

\begin{lem}
Let $\xi(t)$ be a local minimum or maximum of the radius and let $r(t) = \psi(\xi(t),t)$. Then $\lim_{t \to T} r(t)$ exists. If this limit is zero, then for a minimum we have
\begin{gather*}
    -\frac{n-1}{r(t)} \leq r'(t) \leq -\frac{n-1}{2r(t)} + \frac{C}{2n} r(t)\\
    (1-o(1)) \sqrt{2(n-1)(T-t)} \leq r(t) \leq \sqrt{2(n-1)(T-t)} \nonumber
\end{gather*}
and for a maximum we have
\begin{gather*}
    -\frac{C}{r(t)} \leq r'(t) \leq -\frac{n-1}{r(t)}\\
    \sqrt{2(n-1)(T-t)} \leq r(t) \leq K\sqrt{(T-t)} \nonumber
\end{gather*}
for $C$ and $K$ depending only on $g_0$ and $n$.

Therefore, any neck or bump with $r\to 0$ eventually satisfies $\frac{1}{2} \leq u \leq K$, and any such neck satisfies $u<1$ and $u \to 1$.
\end{lem}
\begin{proof}
See Proposition 5.4 and Lemma 6.1 and the Remark after it in \cite{neckpinch}. We can use the bounds from previous Lemma in case of local maximum and say $\frac{-C}{r}\leq \psi_{ss}$.
\end{proof}

An almost linear profile function implies large values of rescaled radius next to the poles. We can show that such regions of large radius move away from the neck exponentially fast.
\begin{lem}\label{caps}
The $u$ value of the left and right-most bumps is bigger than $D \exp(\tau/2)$ for large $\tau$ where $D$ depends only on $g_0$.
\end{lem}
\begin{proof}
If $\psi \to 0$ at the last bump, then $u \leq K$ by the previous Lemma. However, because of almost linear $u$, we know $u$ becomes arbitrary large at the bumps. Hence $\lim_{t\to T} \psi > D > 0$ and therefore $u> D \exp(\tau/2)$.
\end{proof}

Now we can prove a lower bound for $u$.

\begin{lem}\label{ulowbound}
For $A>0$ there exist $\tau_0 = \tau_0(g_0, A)$ such that for $\tau\geq \tau_0$ and $|\sigma| \leq 4 A \sqrt{\tau}$ we have $u \geq \frac{1}{2}$.
\end{lem}
\begin{proof}
Our Lemma \ref{caps} implies there is a first bump on each side of the origin with $\lim \psi >D>0$ and $u>D \exp(\tau/2)$. This together with $|u_\sigma|= |\psi_s/\sqrt{2(n-1)}| \leq C(g_0)$ imply that the region around the right (left) bump where $\frac{D}{2}e^{\tau/2} \leq u$ has length (measured in $\sigma$) at least $c \exp(\tau/2)$ for some positive $c=c(g_0, D)$. So at the right (left) bump we have $|\sigma|\geq c \exp(\tau/2)> 5A \sqrt{\tau}$ for $\tau\geq \tau_0(g_0, D, A)$. In particular the region $|\sigma| \leq 4A\sqrt{\tau}$ does not reach the polar caps. Since there are only finitely many local minima of $u$ and each is converging to infinity or $1$, we can conclude that for $|\sigma|\leq 4 A \sqrt{\tau}$, $u(\tau, \sigma)$ is bounded below by $\frac{1}{2}$.
\end{proof}

We will use the following application of Hamilton-Ivy pinching estimate, which was generalized to conformally flat metric in \cite{hamivy}.
\begin{lem}\label{hiv}
Assume the eigenvalues of Ricci tensor $\lambda=-\frac{\psi_{ss}}{\psi}+ (n-1)\frac{1-\psi_s^2}{\psi^2}$ and $\nu=-n\frac{\psi_{ss}}{\psi}$ satisfy $\lambda, \nu \geq -1$ initially, and assume $\log(1+T)>3$ (We can achieve both by rescaling the initial metric). There is some $C_0$ and $\tau_0$ depending only on $g_0$ with the following property: For $\tau\geq \tau_0$ and at any point, the bounds
$$ 0 < u_{\sigma \sigma} < \frac{1}{4L} \qquad |u_\sigma| \leq \frac{1}{\sqrt{4(n-1)}} \qquad \frac{1}{2} \leq u \leq L$$
for some $L$ imply the stronger bound $$ u_{\sigma \sigma} < \frac{C_0}{\tau}$$
\end{lem}
\begin{proof}
At the given point we have the bound $$|R| \leq \frac{C}{\psi^2} = \frac{C}{(T-t)u^2} \leq \frac{C}{T-t}$$ since we have $u \geq \frac{1}{2}$. The derivative bounds imply
\begin{gather*}
    \nu=-n\frac{\psi_{ss}}{\psi}=-\frac{n}{T-t}\frac{u_{\sigma\sigma}}{u}< 0\\
    \lambda=\frac{1}{2(T-t)} \left(\frac{(1-2(n-1)u_\sigma^2-2uu_{\sigma \sigma}}{u^2} \right) >0
\end{gather*}
So by the Hamilton-Ivy estimate we have
\begin{align*}
    &\frac{C}{T-t} \geq R \geq -\nu \left(\log(-\nu) + \log(1+t) - 3\right) \geq -\nu\log(-\nu) = \\
    &\frac{n}{T-t}\frac{u_{\sigma\sigma}}{u} \log \left( \frac{n}{T-t}\frac{u_{\sigma\sigma}}{u} \right) = \frac{n}{T-t}\alpha (\log n + \tau+ \log \alpha)
\end{align*}
where $\alpha=\frac{u_{\sigma\sigma}}{u}$, and we assume $t$ is close enough to $T$. Now if $\alpha > \frac{\max\{4C/n, 1\}}{\tau}$ and $\tau$ is large enough so that $\frac{\tau}{2}\geq \log \tau$, we get
\begin{gather*}
    \frac{C}{n}\geq \alpha (\tau+ \log \alpha) \geq \alpha (\tau+ \log \frac{\max\{4C/n, 1\}}{\tau}) \geq \alpha (\frac{\tau}{2}) \Rightarrow\\
    \frac{2C/n}{\tau} \geq \alpha
\end{gather*}
which is a contradiction.
\end{proof}

The previous Lemma gives us our first apriori estimates.
\begin{lem}\label{baseptest}
There is some $C_0$ and $\tau_0$ depending only on $g_0$ such that for any neck $\sigma_*(\tau)$ with radius $\psi$ becoming zero at $t=T$ and $\tau \geq \tau_0$, we have 
\begin{gather*}
    u_{\sigma \sigma}(\tau, \sigma_*) \leq \frac{C_0}{\tau}\\
    1-\frac{C_0}{\tau} \leq u(\tau, \sigma_*)
\end{gather*}
\end{lem}
\begin{proof}
Let $u_* = u(\tau, \sigma_*)$. Since $\psi \to 0$, we must have $u_* \to 1$. We may assume $\tau_0$ is large enough so that Lemma \ref{hiv} holds, $\frac{C_0}{\tau} \leq \frac{1}{8}$ in that Lemma, and $\frac{1}{2}\leq u_*\leq 1$. Since $u_*\to1$, $u_{\sigma \sigma}(\tau, \sigma_*)\geq \frac{1}{4}$ for large $\tau$ would imply
\begin{gather*}
    \frac{d}{d\tau} u_*^2 = 2u_*u_\tau(\tau, \sigma_*) = 2u_*u_{\sigma\sigma}(\tau, \sigma_*) + u_*^2-1 \geq \frac{1}{8}
\end{gather*}
which cannot hold for too long since $u_*\leq 1$. But if we ever have $u_{\sigma \sigma} \leq \frac{1}{4}$ at the neck, then Lemma \ref{hiv} is applicable (since $u_\sigma=0$ at the neck) and we get $u_{\sigma\sigma} \leq \frac{C_0}{\tau} \leq \frac{1}{8} < \frac{1}{4}$ from that point on.

This in particular implies
\begin{align*}
    &\frac{d}{d\tau} u(\tau, \sigma_*(\tau)) = u_\tau(\tau, \sigma_*) = \\
    &u_{\sigma \sigma}(\tau, \sigma_*) + \frac{1}{2}(u(\tau, \sigma_*)-\frac{1}{u(\tau, \sigma_*)}) \leq \frac{C_0}{\tau} +\frac{u(\tau, \sigma_*)}{2} -\frac{1}{2} \Rightarrow \\
    &\frac{d}{d\tau} (e^{-\tau/2} u_*) \leq e^{-\tau/2}\left( \frac{C_0}{\tau} - \frac{1}{2}\right) \Rightarrow \\
    &0 - e^{-\tau/2} u_* = \left. e^{-\tau/2} u_* \right\rvert_\tau^\infty \leq \int_\tau^\infty e^{-r/2}\left( \frac{C_0}{r} - \frac{1}{2}\right)\, dr \leq \\
    &\int_\tau^\infty e^{-r/2}\left( \frac{C_0}{\tau} - \frac{1}{2}\right)\, dr = \frac{C_0}{\tau}2e^{-\tau/2} - e^{-\tau/2} \Rightarrow\\
    &1-\frac{C}{\tau} \leq u(\tau, \sigma_*)
\end{align*}
\end{proof}

\subsection{Estimate For $u_\sigma$}

We now move on to the apriori decay estimate for $u_\sigma$. Any point with $|\sigma| \leq A \sqrt{\tau}$ is between the left-most and right-most bumps. So, at any $\tau$, it would either be next to a bump with $u \to \infty$, or between a neck with $u \to 1$ and a bump with $u \leq K$. In the first case we can integrate the estimate in Lemma \ref{hiv} from the nearest neck. But in the second case, the nearest neck might be too far and this idea does not work. We can integrate Lemma \ref{hiv} only up to a distance $o(\tau)$ because the upper bound for $u_\sigma$ is not guaranteed for larger scales.

We use a barrier argument instead. Consider a neck and focus on the region on the right or left side of it between the neck and the next bump. In such a region, we can use $u$ as a variable and compute the evolution equation of $\psi_s^2$ in this new variable. This has been done in multiple previous works, such as \cite{degenneck}. We will however include the proof for completeness. After computing the equation, we construct an upper barrier in terms of $u$. The boundary conditions are controlled using the estimates for $u$ in the previous subsection.

\begin{lem}\label{radeq}
Let $Z(\tau, u)$ be a function such that
$$\psi_s^2(t,x) = 2(n-1)u_\sigma^2 (\tau, \sigma)= Z(\tau, u(\tau, \sigma))$$
The function $Z$ satisfies the following equation $\mathcal{F}[Z]=0$ with commuting variables $u$ and $\tau$:
\small
\begin{gather}
    Z Z_{u u} - \frac{1}{2} Z_u^2 + \frac{n-1-Z}{u} Z_u + \frac{2(n-1)}{u^2} (1-Z)Z -(n-1) u Z_u - 2(n-1)Z_\tau = 0
\end{gather}
\normalsize
\end{lem}
\begin{proof}
We differentiate both sides of the defining relation for $Z$ with respect to $\tau$. To compute the equation for $u_\sigma^2$, we start from the equation for $u$ and differentiate. Let $z=u_\sigma=\frac{1}{\sqrt{2(n-1)}}\sqrt{Z}$. In the non-commuting variables we have
\begin{align*}
    & u_\tau = u_{\sigma \sigma} + \frac{1}{2}\left( u-\frac{1}{u}\right) + (n-1) \frac{u_\sigma^2}{u}\Rightarrow\\
    & z_\tau = z_{\sigma \sigma} + \frac{1}{2}\left( z+\frac{z}{u^2}\right)+ 2(n-1) \frac{zz_\sigma}{u} - (n-1) \frac{z^3}{u^2} - \left(\frac{1}{2}+ n\frac{z_\sigma}{u} \right)z\\
    & = z_{\sigma \sigma} + \frac{1}{2} \frac{z}{u^2}+ (n-2) \frac{zz_\sigma}{u} - (n-1) \frac{z^3}{u^2}\Rightarrow\\
    & Z_\tau = 2(n-1)(z^2)_\tau = Z_{\sigma \sigma} - 4(n-1)(z_\sigma)^2 + \frac{1}{u^2}Z + (n-2)\frac{zZ_\sigma}{u} - \frac{1}{u^2}Z^2
\end{align*}
where we have used $2 z z_{\sigma \sigma}= (z^2)_{\sigma \sigma} - 2(z_\sigma)^2$ and $2z^2z_\sigma= z (z^2)_\sigma$. Now note
\begin{gather*}
    \frac{\partial}{\partial \sigma} = u_\sigma \frac{\partial}{\partial u}  = z\frac{\partial}{\partial u} =\sqrt{\frac{Z}{2(n-1)}} \frac{\partial}{\partial u}
\end{gather*}
So we get $zZ_\sigma = \frac{1}{2(n-1)}ZZ_u$ and
\begin{gather*}
    z_\sigma = \sqrt{\frac{Z}{2(n-1)}} \frac{\partial}{\partial u} \sqrt{\frac{Z}{2(n-1)}} = \frac{1}{4(n-1)}Z_u\\
    Z_{\sigma \sigma} = \frac{1}{2(n-1)}\sqrt{Z} \frac{\partial}{\partial u}\left( \sqrt{Z} Z_u \right) = \frac{1}{2(n-1)} \left(\frac{1}{2}Z_u^2+ZZ_{uu} \right)
\end{gather*}
and hence we have
\begin{gather*}
    2(n-1)Z_\tau = ZZ_{uu} + \frac{2(n-1)}{u^2}Z(1-Z) + (n-2)\frac{ZZ_u}{u}
\end{gather*}

To get the equation in commuting variables, we must add $2(n-1)Z_u u_\tau$ to the right hand side. We have $2u_\sigma u_{\sigma \sigma} = \frac{1}{2(n-1)}Z_\sigma = \frac{1}{2(n-1)} u_\sigma Z_u$, therefore
\begin{align*}
    & u_\tau = u_{\sigma \sigma} + \frac{1}{2}\left( u-\frac{1}{u}\right) + (n-1) \frac{u_\sigma^2}{u}= \frac{1}{4(n-1)}Z_u + \frac{1}{2}\left( u-\frac{1}{u}\right)+ \frac{1}{2}\frac{Z}{u}\Rightarrow\\
    & 2(n-1)Z_u u_\tau = \frac{1}{2}Z_u^2 + (n-1)\left(+ ( u-\frac{1}{u})Z_u+ \frac{ZZ_u}{u}\right)
\end{align*}
Putting everything together, we are done.
\end{proof}

We now construct a super-solution for the equation we derived.
\begin{lem}\label{supersol}
Consider the equation
\begin{equation}
    z_\tau = \frac{1}{2(n-1)}\left ( zz_{uu} - \frac{1}{2} z_u^2 + \frac{n-1-z}{u} z_u + \frac{2(n-1)}{u^2} (1-z)z -(n-1) u z_u\right)
\end{equation}
For every positive number $c$ and any $L>1$, there exist $B_0$ and $\tau_0$ depending on $c$ and $L$ such that for $B \geq B_0$ and $\tau \geq \tau_0$ the function
$$ \bar Z(\tau, u) = B\frac{1}{\tau}\left( 1-\frac{1}{(u+\frac{c}{\tau})^2} \right)$$
is a super-solution for $1-\frac{c}{\tau}\leq u \leq L$.
\end{lem}
\begin{proof}
For notational simplicity, define
\begin{gather*}
    \mathcal{D}[z](\tau, u) = \left. \left( \frac{2}{u^2}z+ (\frac{1}{u}- u) z_u \right) \right \rvert_{(\tau, u)}\\
    \mathcal{Q}[z](\tau, u)= \left. \left(z z_{uu} - \frac{1}{2} z_u^2 - \frac{1}{u} zz_u - \frac{2(n-1)}{u^2} z^2 \right) \right \rvert_{(\tau,u)} 
\end{gather*}
so the equation for $z$ can be written as
\begin{gather*}
    (n-1)\mathcal{D}[z]+\mathcal{Q}[z]-2(n-1)z_\tau =\mathcal{F}[z]=0
\end{gather*}
First let $Z_1(\tau, u) = B\frac{1}{\tau}\left(1-\frac{1}{u^2}\right)$. For any $L_1$ and for $1\leq u\leq L_1$ we have
\begin{gather*}
    (n-1) \mathcal{D}[Z_1](\tau, u) = 0 \\
    \mathcal{Q}[Z_1](\tau, u)= B^2 \frac{1}{\tau^2} \left(\frac{2(1-n)}{u^2} + \frac{4(n-3)}{u^4} + \frac{2(4-n)}{u^6} \right)<-B^2 \frac{1}{\tau^2} C_{L_1} \leq 0
\end{gather*}
for some positive $C_{L_1}$ depending on $n$ and $L_1$. Now assume $\tau \geq 1$ and pick $L_1 >L+c$. We have $\bar Z(\tau, u)= Z_1(\tau, u+\frac{c}{\tau})$, and the assertion holds essentially because a perturbation of order $\frac{1}{\tau}$ creates errors of order at most $\frac{1}{\tau^2}$. 
We compute the contribution of $\mathcal{D}$, $\mathcal{Q}$ and $\bar Z_\tau$ separately. Set $\bar Z=\frac{B}{\tau}H$. We have
\begin{align*}
    & \mathcal{D}[\bar Z](\tau,u)-\mathcal{D}[Z_1](\tau,u+\frac{c}{\tau})\\
    & = (\frac{2}{u^2}-\frac{2}{(u+\frac{c}{\tau})^2})Z+ (\frac{1}{u}-u - \frac{1}{u+\frac{c}{\tau}}+u+\frac{c}{\tau})Z_u\\
    & = \frac{2Bc}{\tau^2}(\frac{2u+\frac{c}{\tau}}{u^2(u+\frac{c}{\tau})^2})H + \frac{Bc}{\tau^2}(\frac{1}{u(u+\frac{c}{\tau})}+1)H_u \leq C\frac{B}{\tau^2}\\
    & \mathcal{Q}[\bar Z](\tau,u) - \mathcal{Q}[Z_1](\tau,u+\frac{c}{\tau}) \\
    & = -(\frac{1}{u}-\frac{1}{u+\frac{c}{\tau}})\bar Z\bar Z_u -2(n-1)(\frac{1}{u^2}- \frac{1}{(u+\frac{c}{\tau})^2})\bar Z^2 \\
    & =-\frac{B^2c}{\tau^3}\frac{1}{u(u+\frac{c}{\tau})}HH_u -2(n-1) \frac{B^2c}{\tau^3}(\frac{2u+\frac{c}{\tau}}{u^2(u+\frac{c}{\tau})^2})H^2 \leq C\frac{B}{\tau^3}\\
    & -\bar Z_\tau(\tau,u)= \left( B\frac{1}{\tau^2}( 1-\frac{1}{( u+\frac{c}{\tau})^2} ) + Bc\frac{1}{\tau^3}\frac{2}{(u+\frac{c}{\tau})^3} \right) \leq C\frac{B}{\tau^2}
\end{align*}
We have used the fact that $H$ and $H_u$ are positive and bounded on $[1-\frac{c}{\tau}, L]$. Putting these together we have
\begin{align*}
    & \mathcal{F}[\bar Z](\tau,u) = (n-1)\mathcal{D}[\bar Z](\tau,u)+\mathcal{Q}[\bar Z](\tau,u)-2(n-1)\bar Z_\tau(\tau,u) = \\
    & (n-1)\left( \mathcal{D}[\bar Z](\tau,u)-\mathcal{D}[Z_1](\tau,u+\frac{c}{\tau}) \right)+ \left( \mathcal{Q}[\bar Z](\tau,u) - \mathcal{Q}[Z_1](\tau,u+\frac{c}{\tau})\right)\\
    & +(n-1)\mathcal{D}[Z_1](\tau,u+\frac{c}{\tau})+\mathcal{Q}[Z_1](\tau,u+\frac{c}{\tau})-2(n-1)\bar Z_\tau(\tau,u) \leq \\
    & C\frac{B}{\tau^2} + C\frac{B}{\tau^3} - \frac{B^2}{\tau^2}C_{L_1}^2
    % & (n-1)(\frac{2}{u^2}-\frac{2}{(u+\frac{c}{\tau})^2})z+ (n-1)(\frac{1}{u}-u - \frac{1}{u+\frac{c}{\tau}}+u+\frac{c}{\tau})z_u + \\
    % & -(\frac{1}{u}-\frac{1}{u+\frac{c}{\tau}})zz_u -2(n-1)(\frac{1}{u^2}- \frac{1}{(u+\frac{c}{\tau})^2})z^2 -B^2 \frac{1}{\tau^2}C_{L_1}^2\\
    % & + 2(n-1) \left( B\frac{1}{\tau^2}( 1-\frac{1}{( u+\frac{c}{\tau})^2} ) + Bc\frac{1}{\tau^3}\frac{2}{(u+\frac{c}{\tau})^3} \right) = \\
    % & (n-1)\frac{2Bc}{\tau^2}(\frac{2u+\frac{c}{\tau}}{u^2(u+\frac{c}{\tau})^2})h + (n-1)\frac{Bc}{\tau^2}(\frac{1}{u(u+\frac{c}{\tau})}+1)h_u \\
    % & -\frac{B^2c}{\tau^3}\frac{1}{u(u+\frac{c}{\tau})}hh_u -2(n-1) \frac{B^2c}{\tau^3}(\frac{2u+\frac{c}{\tau}}{u^2(u+\frac{c}{\tau})^2})h^2 -B^2 \frac{1}{\tau^2}C_{L_1}^2\\
    % & + 2(n-1) \frac{B}{\tau^2}h + 2(n-1)\frac{Bc}{\tau^3}h_u \leq C\frac{B}{\tau^2} + C\frac{B}{\tau^3} - \frac{B^2}{\tau^2}C_{L_1}^2
\end{align*}
for some $C$ depending on $c$, $n$ and $L$. Since $C_{L_1}>0$, by picking $B$ and then $\tau$ large depending on $c$, $n$ and $C_{L_1}$ we can make the last expression negative.
\end{proof}

\begin{lem}\label{barrier}
Take a neck with $u \to 1$ and a bump with $1\leq u \leq K$, and consider the region between them on one side of the neck. Let $C_0$ be as in Lemma \ref{baseptest} for the neck, and suppose $$Z(\tau, u(\sigma, \tau)) = \psi_s^2$$
in this region. Then there exists $B$ and $\tau_0$ depending only on $g_0$ such that $\tau \geq \tau_0$ implies
\begin{gather*}
    Z(\tau, u) \leq B\frac{1}{\tau}\left( 1-\frac{1}{(u+\frac{2C_0}{\tau})^2} \right)
\end{gather*}
\end{lem}
\begin{proof}
In Lemma \ref{supersol}, take $L=K$ and $c=2C_0$ of Lemma \ref{baseptest}. This give some $B_0$ and $\tau_0$. By taking $\tau_0$ even larger, we may assume the bound of Lemma \ref{baseptest} at the neck and the bound $1 \leq u \leq K$ at the bump hold. Now take $B$ larger than $B_0$, and large enough so that $\bar Z(\tau, u) = B\frac{1}{\tau}\left( 1-\frac{1}{(u+\frac{C_0}{\tau})^2} \right)$ satisfies $Z < \bar Z$ at $\tau_0$.

Assume there is a first time $\tau_1$ such that $Z< \bar Z$ fails, so that $Z(\tau_1, u_*)=\bar Z(\tau_1, u_*)$ for some $u_*$. $Z$ is zero at the neck and the bump, and by Lemma \ref{baseptest} we have $1-\frac{C_0}{\tau} \leq u$ at the neck. However $\bar Z$ is zero only at $1-\frac{2C_0}{\tau}$. So $u_*$ cannot be at the boundary. But $u_*$ being an interior point and first intersection implies
\begin{gather*}
    Z_u(\tau_1, u_*)=\bar Z_u(\tau_1, u_*) \,\, , \,\,  Z_{uu}(\tau_1, u_*) \leq \bar Z_{uu}(\tau_1, u_*) \,\, , \,\, Z_\tau (\tau_1, u_*) \geq \bar Z_\tau(\tau_1, u_*)
\end{gather*}
which gives
$$0 = \mathcal{F}[Z](\tau_1, u_*) \leq \mathcal{F}[\bar Z](\tau_1, u_*) <0$$
which is a contradiction. Therefore $Z < \bar Z$ for all $\tau \geq \tau_0$. Since there are only finitely many necks and bumps, and $K$ and $C_0$ of Lemma \ref{baseptest} depend only on $g_0$, we may choose $B$ and $\tau_0$ depending only on $g_0$.
\end{proof}

Now we can state our decay estimate for $u_\sigma$.
\begin{prop}\label{basicderestim}
For any $A$ there exist $\tau_0 = \tau_0(g_0, A)$ and $C_A=C(g_0, A)$ such that for $\tau \geq \tau_0$
\begin{gather*}
    |u_\sigma| \leq \frac{C_A}{\sqrt{\tau}} \quad \text{for}\,\, |\sigma| \leq 4A\sqrt{\tau}
\end{gather*}
\end{prop}
\begin{proof}
Let $\sigma_*(\tau)$ and $\sigma^*(\tau)$ be the closest bumps $\sigma(\tau)$ to the origin such that $u(\tau, \sigma(\tau))\to \infty$. Assume $\tau_0$ is large enough so that Lemmas \ref{ulowbound}, \ref{hiv}, \ref{baseptest} and \ref{barrier} hold. Also take it large enough so that $\tau\geq \tau_0$ implies
\begin{gather*}
    \frac{C_0}{\tau} \leq \frac{1}{4}\\
    \frac{4C_0A}{\sqrt{\tau}} \leq \frac{1}{2\sqrt{4(n-1)}}
    \frac{C_0}{\tau} \leq \frac{1}{4(2+16C_0A^2)}
\end{gather*}
for $C_0$ in Lemma \ref{baseptest}.

Let $$\underline{\sigma}_1(\tau) < \overline{\sigma}_1(\tau) < \underline{\sigma}_2(\tau)< \dots < \overline{\sigma}_k(\tau) < \underline{\sigma}_{k+1}(\tau)$$
be the sequence of all necks and bumps that lie between $\sigma_*(\tau)$ and $\sigma^*(\tau)$. We may assume $\frac{1}{2}\leq u \leq 1$ for any neck and $\frac{1}{2} \leq u \leq K$ for any bump in the sequence. For $\tau \geq \tau_0$, any point with $|\sigma|\leq 4A\sqrt{\tau}$ lies between $\sigma_*$ and $\sigma^*$. There are two possibilities.

First, assume $\sigma \geq \underline{\sigma}_{k+1}(\tau)$ or $\sigma \leq \underline{\sigma}_1(\tau)$. We deal with the first case and the second is similar. The assumptions of Lemma \ref{hiv} hold at the neck $\underline{\sigma}_{k+1}$ since $u\leq 1$ and $u_{\sigma \sigma} \leq \frac{C_0}{\tau} \leq \frac{1}{4}$ by Lemma \ref{baseptest} and our choice. As long as they hold for $L=2+16C_0A^2$, we get $u_{\sigma\sigma} \leq \frac{C_0}{\tau}<\frac{1}{4L}$ and by our choice
\begin{gather*}
    u_\sigma(\tau, \sigma) = \int_{\underline{\sigma}_{k+1}}^\sigma u_{\sigma \sigma} \, d\sigma \leq 4A \sqrt{\tau} \frac{C_0}{\tau} = \frac{4C_0A}{\sqrt{\tau}} < \frac{1}{\sqrt{4(n-1)}}\\
    u(\tau, \sigma) = u(\tau, \underline{\sigma}_{k+1}) + \int_{\underline{\sigma}_{k+1}}^\sigma u_\sigma \, d\sigma \leq 1+ 16C_0 A^2 < L
\end{gather*}
since $\sigma- \underline{\sigma}_{k+1} \leq 4A\sqrt{\tau}$. This shows that there cannot be a first point for the assumption in Lemma \ref{hiv} to fail within $|\sigma| \leq 4A \sqrt{\tau}$

The second possibility is  $\underline{\sigma}_i(\tau) \leq \sigma \leq \overline{\sigma}_i(\tau)$ or $ \overline{\sigma}_i(\tau)  \leq \sigma \leq \underline{\sigma}_{i+1}(\tau)$. Between any consecutive neck and bump with $\frac{1}{2}\leq u \leq K$ we have
\begin{gather*}
    \psi_s^2 = 2(n-1)u_\sigma^2 \leq B\frac{1}{\tau}\left( 1-\frac{1}{(u+\frac{2C_0}{\tau})^2} \right) \leq \frac{C}{\tau}
\end{gather*}
\end{proof}

\section{Asymptotic Analysis of The Profile}\label{asympt}

As mentioned before, we choose to work with the quantity
$$f=\frac{u_\sigma}{u}$$ 
and follow the strategy of Filippas and Kohn described in section \ref{outline}. We start with localizing the equation for $f$ derived in Lemma \ref{ufeq} and state the precise estimates needed at each step. Initially, we only need very weak control of $u$, $f$ and $f_\sigma$ to implement the strategy, but our estimates can be improved if we have better control. We record these sharper versions as Corollaries after each estimate, since we will need them later. As discussed before, in order to deal with the exponentially small error terms coming from cut-offs, we use a technical modification of the Merle-Zaag ODE lemma. This is Lemma \ref{myMZ}, and it is proven in the Appendix. We postpone some of the more tedious proofs and computations to the end of the section in order to keep the flow of the arguments.

We will also discuss the following two key steps. First, we show that in the case of exponential decay, the optimal decay rate is either infinity (faster than any exponential), or an eigenvalue $\lambda_m$ of the operator $\mathcal{A} = \mathcal{L}-\frac{1}{2}$. Furthermore, the modes corresponding to smaller eigenvalues (with potentially slower decay rates) can be controlled. Second, we show how to deal with the constant mode of $u-1$ (in the direction of $h_0$). Since modes of a function and its derivative are related, we can control the ODE for the first mode of $U=\log u$ in terms of the modes of $f=U_\sigma$. This allows us to show that if the constant mode of $U$ is not small compared to $\|f\|$, then $u-1$ grows exponentially. We give the proof of our main Theorem \ref{main}, assuming the contents of the following sections.

From now on, we fix a smooth function $\chi(r)$ such that 
\begin{gather*}
\text{$\chi(r)=1$ for $|r|\leq 1$, \quad $\chi(r)=0$ for $|r| \geq 2$, \quad $r\chi'(r)\leq 0$}    
\end{gather*}
All the norms and inner products are in the Gaussian weighted $L^2$ space $$\mathfrak{H}=L^2(\rho \, d\sigma), \quad \rho(\sigma)= \exp(\frac{-\sigma^2}{4})$$
We also assume, without further mention, that $\tau\geq \tau_0$ where $\tau_0$ is large enough so that the previous results (Lemmas \ref{ulowbound} and \ref{baseptest} and Proposition \ref{basicderestim}) hold.

\subsection{Asymptotic Behavior of $f$}

Before anything, we point out that since $u \geq \frac{1}{2}$ on $|\sigma| \leq 4A\sqrt{\tau}$, we have the following bounds.
\begin{gather}
    |f| = |\frac{u_\sigma}{u}| \leq \frac{C_A}{\sqrt{\tau}} \leq M\\
    |f_\sigma| =|\frac{uu_{\sigma\sigma}-u_\sigma^2}{u^2}| \leq N
\end{gather}

First we state the evolution equation satisfied by $f$ and its localization for easier reference.

\begin{lem}
The quantity $f=\frac{u_\sigma}{u}$ satisfies the equation
\begin{gather}
    f_\tau= \mathcal{A}f -n\left( \int_0^\sigma f^2(\xi)\, d\xi  \right)f_\sigma + \left( \frac{1}{u^2}-1 \right)f -nf^3
\end{gather}
where the operator $\mathcal{A}$ is defined as
\begin{gather} \label{bdef}
    \mathcal{A} = \frac{\partial^2}{\partial \sigma^2} - \frac{\sigma}{2} \frac{\partial}{\partial \sigma} + \frac{1}{2}
\end{gather}

For any cut-off function $\eta$ we have
\begin{gather} \label{feqloc}
    (f\eta)_\tau= \mathcal{A}(f\eta) -n\left( \int_0^\sigma f^2(\xi)\, d\xi  \right) f_\sigma \eta+ \left( \frac{1}{u^2}-1 \right)f\eta -nf^3\eta + E(\tau, \sigma)
\end{gather}
where
\begin{gather}\label{eeq}
    E(\tau,\sigma) = f\eta_\tau - f\eta_{\sigma \sigma} -2 f_\sigma \eta_\sigma + \frac{\sigma}{2} f \eta_\sigma
\end{gather}
\end{lem}

Our first main estimate is the following.
\begin{prop} \label{feqlocest}
For a fixed large positive number $A$, define the cut-off functions $$\eta(\tau,\sigma)=\chi(\frac{\sigma}{A\sqrt{\tau}}) \qquad \text{and} \qquad \theta(\tau, \sigma)=\chi(\frac{\sigma}{2A\sqrt{\tau}})$$ For every $\varepsilon>0$, there exist $\tau_1=\tau_1(\tau_0, \varepsilon)$ and constants $C_0=C(g_0)$ and $C_\varepsilon=C(g_0, \varepsilon)$ such that the following holds: The function $f\eta$ satisfies
$$
(f\eta)_\tau= \mathcal{A}(f\eta) + F(\tau, \sigma) + E(\tau, \sigma)
$$
where the operator $\mathcal{A}$ and functions $F$ and $E$ are defined in \ref{bdef}, \ref{feqloc} and \ref{eeq}, and for $\tau\geq \tau_1$ we have the estimates
\begin{gather}
\|F\|^2(\tau) \leq C_0 \varepsilon \int f^2 \eta^2 \rho + C_\varepsilon \int f^4 \theta^4 \rho + C_0 \exp(- \frac{A^2}{8} \tau)\nonumber \\
\|E\|^2(\tau) \leq C_0 \exp(- \frac{A^2}{8} \tau)
\end{gather}
\end{prop}
By examining the proof, which is presented at the end of this section, we get a more detailed variant.
\begin{cor} \label{reffeqlocest}
For $\tau\geq \tau_0$, $\|F\|^2$ is less than
\begin{small}
\begin{multline}
    (C \, \delta(\tau)^2 + C_0 \varepsilon^2(\tau)) \int f^2 \eta^2 \rho \, d\sigma +\\
    (C_0 + \frac{C_0 N^2}{\varepsilon(\tau)^2} + C(M^2+N^2)) \int f^4 \theta^4 \rho \, d\sigma + C_0 M^4 \exp(- \frac{A^2}{8} \tau) \\
\end{multline}
\end{small}
Here $\delta(\tau) = \frac{1}{u(\tau, 0)^2}-1$, $\varepsilon(\tau)$ is arbitrary and we have the bounds $|f|\leq M$ and $|f_\sigma|\leq N$ on the support of $\theta$.
\end{cor}

%From now on, we fix some number $A>1$ and assume $\tau\geq 1$. We let $\chi$, $\eta$ and $\theta$ be as described in the proposition above and assume the bounds of Lemma \ref{bound} on the support of $\theta$.

As discussed in section \ref{outline}, let
\begin{equation}
    I^2= \int f^4 \theta^4 \sigma^k \rho
\end{equation}
for some large $k$ to be chosen. For any $\varepsilon>0$ and $\delta>0$, there exists $\tau_{\varepsilon\delta}= \tau_{\varepsilon\delta}(\varepsilon, \delta)$ such that $\tau\geq \tau_{\varepsilon\delta}$ implies $|f|=|\frac{u_\sigma}{u}|<\varepsilon$ for $|\sigma|\leq \delta^{-1}$. So for $\tau\geq \tau_{\varepsilon\delta}$ we have
\begin{equation} \label{quadleqi}
    \int f^4 \theta^4 \rho = \int_{|\sigma|\leq \delta^{-1}} f^4 \theta^4 \rho+ \int_{|\sigma|> \delta^{-1}} f^4 \theta^4 \rho \leq \varepsilon^2 \int f^2 \theta^2 \rho + \delta^k I^2
\end{equation}
We will need the following more precise version of this inequality as well.
\begin{cor} \label{refquadleqi}
We may replace $\delta$ with any function $\delta(\tau)$ that becomes small enough (such as small constants) and replace $\varepsilon$ with an upper bound for $|f|$ on the set $|\sigma| \leq \delta(\tau)^{-1}$.
\end{cor}

We have the following estimate for $I$, proven at the end of this section.
\begin{prop}\label{idifineq}
Suppose $\tau \geq \tau_0$ so that and $u\geq \frac{1}{2}$, $|f| \leq \frac{C}{\sqrt{\tau}} \leq M$ and $|f_\sigma|\leq N$ for $|\sigma|\leq 4A \sqrt{\tau}$. For any positive $a>0$, there is a large even integer $k$ such that the quantity $I$ has the following property: For every $\varepsilon_1>0$ there exists $\tau_2=\tau_2(\tau_0, k, A, \varepsilon_1)$ so that $I$ satisfies the differential inequality
\begin{equation}
    \frac{d}{d\tau} I \leq -a I + \varepsilon_1 \| f\theta \| + C_0 \exp(- \frac{1}{32} A^2\tau)
\end{equation}
for $\tau \geq \tau_2$. Here $C_0 = 2n M^4+ C^2(M+N)^2$ for some universal $C$. (Note that we have the same estimate with $\theta$ replaced with $\eta$ at the expense of an exponential term. See the proof of Proposition \ref{syswexp}.)
\end{prop}
The proof of this proposition actually shows the following refined variant which will be needed later.
\begin{cor}\label{refidifineq}
Instead of $\varepsilon_1$ we may take $\varepsilon(\tau) \delta(\tau)^{2-\frac{k}{2}}$ (if this quantity goes to zero), where $\varepsilon(\tau)$ is an upper bound for $|f|$ on $|\sigma| \leq \delta(\tau)^{-1}$ and $\delta(\tau)$ can be made sufficiently small.
\end{cor}

We now consider projections on eigenspaces. The operator $\mathcal{A}$ is self-adjoint on the Hilbert space $\mathfrak{H}=L^2(\rho\, d\sigma)$. Its spectrum consists only of eigenvalues
\begin{gather*}
\text{$-\lambda_m=\frac{1}{2}-\frac{m}{2}$ for $m=0,1,2, \dots$},     
\end{gather*}
and each eigenvalue $-\lambda_m$ has a one dimensional eigenspace $\mathfrak{H}_m$ spanned by the \\{$m$-th} modified Hermite polynomial $h_m(\sigma)=c_m H_m(\frac{\sigma}{2})$, normalized to have unit norm. Let $\mathfrak{H}_+$, $\mathfrak{H}_0$, $\mathfrak{H}_-$ denote the subspaces corresponding to positive, null and negative eigenvalues. Similarly, let $\mathfrak{H}_{\geq m}$, $\mathfrak{H}_m$, $\mathfrak{H}_{\leq m}$ be the subspaces spanned by $h_j$ for $j\geq m$, $j=m$ and $j \leq m$ respectively. For any symbol ${\alpha \in \{ +, -, m, \geq m, \leq m \}}$, let $\pi_\alpha$ be the projection on $\mathfrak{H}_\alpha$.

Define the quantities
\begin{gather*}
    x(\tau)^2 = \| \pi_+(f \eta) \|^2 =  \langle f\eta, h_0 \rangle^2\\
    y(\tau)^2 = \| \pi_0(f \eta) \|^2 =  \langle f\eta, h_1 \rangle^2\\
    z(\tau)^2 = \| \pi_-(f \eta) \|^2 =  \sum_{m=2}^\infty \langle f\eta, h_m \rangle^2\\
    \zeta(\tau)= z(\tau) + I(\tau)
\end{gather*}
where $\eta(\tau, \sigma)= \chi(\frac{\sigma}{A \sqrt{\tau}})$ as before. By taking $a>\frac{1}{2}$ in Proposition \ref{idifineq} and using inequality \ref{quadleqi} we get the following Proposition. The proof is simple once we have the previous estimates and is included at the end of this section.
\begin{prop} \label{syswexp}
    For any $\varepsilon$ and $\tau \geq \tau_1, \tau_2$ from Propositions \ref{feqlocest} and \ref{idifineq}, we have the system of inequalities
    \begin{gather}\label{odesyswexp}
    \frac{dx}{d\tau}  \geq \frac{1}{2}x - \varepsilon(x+y+\zeta)- C_0\exp(-cA^2\tau) \nonumber \\
    |\frac{dy}{d\tau} | \leq  \varepsilon(x+y+\zeta) + C_0\exp(-cA^2\tau)\\
    \frac{d \zeta}{d\tau}  \leq -\frac{1}{2}\zeta + \varepsilon(x+y+\zeta)+C_0\exp(-cA^2\tau) \nonumber
\end{gather}
for some $C_0$ and some universal $c$. 
\end{prop}

We now present the following adaptation of the well known ODE lemma of Merle and Zaag \cite{merlezaag}. The proof is given in the Appendix.
\begin{lem} \label{myMZ}
    Let $x, y,z$ be non-negative functions of $\tau$. Assume for any $\varepsilon$, there exists $\tau_\varepsilon$ such that for $\tau \geq \tau_\varepsilon$ we have
    \begin{gather}
    \frac{dx}{d\tau}  \geq \frac{1}{2}x - \varepsilon(x+y+z)- B\exp(-b\tau) \nonumber \\
    |\frac{dy}{d\tau} | \leq  \varepsilon(x+y+z) + B\exp(-b\tau)\\
    \frac{d z}{d\tau}  \leq -\frac{1}{2}z + \varepsilon(x+y+z)+B\exp(-b\tau) \nonumber
\end{gather}
where $B$ and $b \gg 1$ are fixed. Then exactly one of the following holds: $(1)$ We have $x \geq Ce^{\delta \tau}$ for some $\delta$; $(2)$ We have $x+z = o(y)$, in which case $ c_\delta e^{-\delta \tau} \leq y \leq C_\delta e^{\delta \tau}$ for any $\delta$; $(3)$ we have $x+y \leq C(z + \exp(-b\tau))$ for some $C$, in which case $x+y+z \leq C_\delta\exp(-(\frac{1}{2}-\delta)\tau)$ for any $\delta$.
\end{lem}

Using our ODE Lemma, we can analyze the projections of $f \eta$. Note that, by making $c$ smaller, we may assume the coefficient $C_0$ in the exponential term corresponding to $B$ in our ODE lemma is controlled independent of the size of the cut-off $A$. Note that the bound $|f|\leq \frac{C}{\sqrt{\tau}}$ for ${|\sigma|\leq 4A\sqrt{\tau}}$ implies
\begin{gather*}
    \| f\eta \|^2 = x^2 + y^2 + z^2 \leq \frac{C}{\tau} \to 0\\
    I^2 = \int f^4 \theta^4 \sigma^k \rho \leq \frac{C_k}{\tau^2} \to 0
\end{gather*}
Therefore the first case in Lemma \ref{myMZ} cannot happen. Applying the ODE lemma, we get the following.

\begin{prop} \label{MZ}
For any $A$ large enough, there exists $\tau_3$ with the following property: For $\tau \geq \tau_3$, we either have $x+\zeta = o(y)$, or $x+y \leq C \left( \zeta+  \exp(-cA^2\tau) \right)$. In the second case, ${\| f\eta \|+I \leq C \exp \left( - \alpha \tau \right)}$ for any $\alpha < \frac{1}{2}$.
\end{prop}

We can also combine our ODE lemma with the extra structure of the modes to get the following lower bound on the decay rate of $\| f\eta \|$.
\begin{prop}
    Assume that for some (and hence every large enough) $A \gg 1$, there exists a sequence of times $\tau_i \to \infty$ and some constant $C_A$ such that
    \begin{gather} 
        \| f(\tau_i, \sigma) \chi(\frac{\sigma}{A\sqrt{\tau_i}}) \| \geq C_A \exp(-cA^2 \tau_i)
    \end{gather}
    Then there exists $A_0$ and $\Lambda$ such that if $A \geq A_0$, 
    \begin{gather}\label{atmostexpdecay}
        \| f\eta \|(\tau) = \| f(\tau, \sigma) \chi(\frac{\sigma}{A\sqrt{\tau}}) \| \geq C_0 \exp(-\Lambda \tau)
    \end{gather}
    for $\tau \geq \tau_*$ and some $C_0$. We refer to \ref{atmostexpdecay} as Decay Condition.
\end{prop}
\begin{rem}
    The above Proposition allows us to replace the exponentially small error terms with $\varepsilon\| f\eta \|$ and use the usual Merle-Zaag lemma. At first glance, this may seem to simplify things. However, due to presence of $I$ in $\zeta = z+I$, even with the stronger conclusion $x+y\leq \varepsilon\zeta$ we will need new techniques to control the non-dominant directions in $z$. Our arguments would not really change with this stronger bound. The only important case is case 2 in our ODE lemma, which is identical to Merle-Zaag's result. We will however require $A \gg \Lambda$ from now on. Note that since $\Lambda$ is fixed, we can pick $A$ at the beginning and fix it throughout.
\end{rem}
\begin{proof}
Fix a large enough $A$ so that the assumption holds, and choose $m$ large such that $\lambda_m=\frac{m-1}{2} \gg 2cA^2$. 
Define
\begin{gather*}
    p(\tau)^2 = \| \pi_{< m}(f \eta) \|^2 =  \sum_{k=0}^{m-1}\langle f\eta, h_k \rangle^2\\
    q(\tau)^2 = \| \pi_{m}(f \eta) \|^2 =  \langle f\eta, h_{m} \rangle^2\\
    r(\tau)^2 = \| \pi_{>m}(f \eta) \|^2 =  \sum_{k=m+1}^\infty \langle f\eta, h_k \rangle^2
\end{gather*}
It is easy to see (c.f. proof of Proposition \ref{syswexp}) that for any $\varepsilon$ these quantities eventually satisfy
\begin{gather*}
    \frac{dp}{d\tau}  \geq \frac{m-2}{2}p - \varepsilon(p+q+r) - \varepsilon I- C_0 \exp(-cA^2 \tau) \\
    |\frac{dq}{d\tau} -\frac{m-1}{2}q| \leq  \varepsilon(p+q+r) + \varepsilon I+ C_0 \exp(-cA^2 \tau) \\
    \frac{dr}{d\tau} \leq -\frac{m}{2}r + \varepsilon(p+q+r) + \varepsilon I + C_0 \exp(-cA^2 \tau)
\end{gather*}

Now let $P(\tau) = e^{\frac{m-1}{2} \tau}p$, define $Q$ and $R$ similarly, and let $J = e^{\frac{m-1}{2}\tau}I$. So if $a$ in Proposition \ref{idifineq} is chosen large enough to ensure $-a + \lambda_m \leq -\frac{1}{2}$ and $m$ is large enough so that $b=\lambda_m -cA^2 \gg 1$, we have the inequality
\begin{gather*}
    \frac{d}{d\tau} J = e^{\frac{m-1}{2}\tau}(\frac{d}{d\tau}I + \frac{m-1}{2}I ) \leq -\frac{1}{2} J + \varepsilon e^{\frac{m-1}{2}\tau} \| f\eta \| +C_0 \exp(-b \tau)
\end{gather*} 
Therefore, $P, Q$ and $S=R+J$ satisfy the conditions of our ODE lemma.

If case 3 of the ODE lemma holds, then we will have
\begin{gather*}
    e^{\frac{m-1}{2}\tau}\| f\eta \| \leq P+Q+S \leq C \exp(-\frac{1}{10}\tau) \Rightarrow \| f\eta \| \leq Ce^{-\frac{m-1}{2}\tau}
\end{gather*}
But this is contradicts the assumption. Hence we must have either case 1 or case 2, which imply a lower bound on $\| f\eta \|$. Since we have a lower bound for a particular choice of $A$, the same bound holds for all larger $A$ as well.

\end{proof}

If we have exponential decay in Proposition \ref{MZ} and the Decay Condition mentioned above, we can say more about the decay rate.
\begin{prop} \label{MZ2}
Assume the Decay condition \ref{atmostexpdecay} holds and $x+y \leq C( \zeta + \exp(-cA^2\tau)$ in Proposition \ref{MZ}, and assume $A$ in Proposition \ref{idifineq} is large enough (compared to $\Lambda$ in the Decay Condition \ref{atmostexpdecay}). Let
$$\lambda = \sup \{ \alpha \geq 0 \, : \, \| f\eta \|+I \leq C_\alpha \exp\left( -\alpha \tau\right) \,\, \text{if} \,\, \tau \geq \tau_\alpha\,, \,\, \text{for some} \,\, C_\alpha \,\, \text{and} \,\, \tau_\alpha \}$$
Then $\lambda = \lambda_m$ for some $m \geq 2$. Moreover,
\begin{gather*}
    \| \pi_{< m} (f\eta)\| \leq C \left( \| \pi_{\geq m} (f\eta)\| +I + \exp(-cA^2 \tau)\right)\\
    \|f\eta\| +I \leq C \exp \left( - (\lambda_m-\delta)\tau \right) \,\, \text{for any} \,\,\, \delta
\end{gather*}
\end{prop}

\begin{proof}
The proof is similar to the previous proposition. First, we have ${\lambda<C\Lambda < \infty}$ because of the Decay Condition. Choose $m$ such that
$$\frac{m-1}{2}< \lambda \leq \frac{m}{2}$$
and define $P, Q, R, S=R+e^{\frac{m-1}{2}\tau}I$ as before. We will show $\lambda = \lambda_{m+1}=\frac{m}{2}$ which gives the assertion with a different choice of $m$.

Note that by our choice of $m$, $A$ being large compared to $\Lambda$ (and hence $\lambda$) and the definition of $\lambda$, for some $\frac{m-1}{2}< \alpha \leq \lambda$ we have
\begin{gather*}
    P+Q+S \leq  C e^{\frac{m-1}{2}\tau} (\| f\eta \|+I) \leq C e^{(\frac{m-1}{2}-\alpha)\tau} \to 0
\end{gather*}
Thus the first case in the ODE lemma is not possible, and we have either $P+S=o(Q)$ or $P+Q\leq C (S+\exp(-b\tau))$. The first case cannot happen, since it would imply that for some $\delta>0$ we have
\begin{gather*}
    C e^{-\delta \tau} \leq e^{\frac{m-1}{2}\tau}q= Q \Rightarrow C\exp(-(\frac{m-1}{2} + \delta)\tau) \leq q \leq C \exp( -(\lambda-\delta) \tau )
\end{gather*}
which is not possible. We conclude $P+Q \leq C(S+\exp(-b\tau))$ and
\begin{gather*}
    \| f\eta \|+ I \leq C \exp(-\frac{m-1}{2}\tau) (S+\exp(-b\tau) ) \leq C \exp(-(\frac{m}{2}-\delta)\tau)
\end{gather*}
for any $\delta$. This also shows $\lambda= \lambda_{m+1} = \frac{m}{2}$.

\end{proof}

\subsection{Controlling $u-1$}

We now deal with controlling $u-1$. When we differentiate $U$, its projection on the subspace spanned by $h_0$ (corresponding to eigenvalue $1$ for $\mathcal{L}$) is discarded. So we need to control this mode separately. We fix $A$ large enough so that the results of previous sections are applicable.

The equation for $$U =\log u$$ was computed in Lemma \ref{ufeq}. Localizing this equation gives
\begin{gather}
    (U\eta)_\tau = \mathcal{L}(U\eta) + \mathcal{N}(U)\eta - \mathcal{G}\eta + \mathcal{E}
\end{gather}
where
\begin{gather}
    \mathcal{N}(U) = \frac{1}{2}\left( 1-e^{-2U} -2U\right)\\
    \mathcal{G}= n\left( \int_0^\sigma f^2 \right) f \nonumber
\end{gather}
and the cut-off error term $\mathcal{E}$ is defined similar to before. Note that $u(\tau, 0) \to 1$ and $|f\theta| \leq \frac{C_A}{\sqrt{\tau}}$ for large $\tau$ implies the upper bounds 
\begin{gather*}
    |U\eta| \leq C_A\\
    \mathcal{N}(U) \leq C_A |U|^2
\end{gather*}

Let
\begin{gather}
    b(\tau) = \langle U\eta \, , \, h_0 \rangle    
\end{gather}
We can control the ODE for $b$ using $\| f\eta \|$. Note that Proposition \ref{MZ2} only provides control with $\| \pi_{\geq m} (f\eta)\| +I$, and not $\| \pi_{\geq m} (f\eta)\|$. Hence the first assertion of the following lemma is actually needed. 

\begin{lem}\label{bode}
For any $\varepsilon$ there exists large $\tau_0$ such that for $\tau \geq \tau_0$ we have
\begin{gather*}
    \frac{d}{d\tau}b^2 \geq (1-\varepsilon^2)b^2 - C_A \|f \eta\|^2 - C\exp(-cA^2 \tau)
\end{gather*}
where $C_A$ is a constant depending on $A$. Furthermore, consider the eigenfunction expansion $f\eta(\tau, \sigma)= \sum_{k=0}^\infty a_k(\tau) h_k(\sigma)$ and assume that for some $m\geq 1$, we have the following Domination Condition
\begin{gather}\label{onemodedoms}
    \lim_{\tau \to \infty} \frac{\sum_{k\neq m} a_k^2}{a_m^2} = 0
\end{gather}
Then we have the better estimate
\begin{gather*}
    \frac{d}{d\tau}b^2 \geq (1-\varepsilon^2)b^2 - \varepsilon^2 \|f \eta\|^2 - C\exp(-cA^2 \tau)
\end{gather*}
for large $\tau$.
\end{lem}

\begin{proof}
We need the following simple property of (modified) Hermite polynomials (see Lemma \ref{recurs2} below): For any function $F$ with appropriate decay and any $m\geq 1$ we have
\begin{gather}\label{recurs}
    \langle F \, , \, h_m \rangle = \sqrt{\frac{2}{m}} \langle F_\sigma \, , \, h_{m-1} \rangle
\end{gather}

The function $b^2(\tau)$ satisfies the equation
\begin{gather*}
    \frac{d}{d\tau}b^2 = 2bb_\tau = 2b^2 + 2 b \big(\langle \mathcal{N}(U)\eta, h_0 \rangle - \langle \mathcal{G}\eta, h_0 \rangle + \langle \mathcal{E} , h_0 \rangle \big)
\end{gather*}
Recall the bound $|f|=|\frac{u_\sigma}{u}|=|U_\sigma| \leq \frac{C_A}{\sqrt{\tau}}$ for $|\sigma| \leq 2A\sqrt{\tau}$ and large $\tau$. This gives 
\begin{gather*}
    \int_0^\sigma f^2 \leq 2A\sqrt{\tau} \frac{C_A}{\tau} \leq \varepsilon \Rightarrow \| \mathcal{G}\eta \| \leq \varepsilon \|f\eta\|
\end{gather*}
We also have $\| \mathcal{E} \| \leq C \exp(-cA^2 \tau)$ as always.

To control $\mathcal{N}(U)$ we proceed as follows. Write
$$U\eta = b(\tau)h_0 + V$$
We can use equality \ref{recurs} and the fact that $U$ is bounded on the support of $\eta$ to write
\begin{align*}
    & \|V\|^2 = \|U\eta\|^2 - b(\tau)^2 = \sum_{m\geq 1} |\langle U \eta, h_m \rangle|^2 \leq C (\| f\eta\|^2+ \|U\eta_\sigma \|^2)\\
    & \leq C \| f\eta\|^2+ C\exp(-cA^2 \tau)
\end{align*}
Now pick $R$ large enough so that if $0\leq \phi \leq 1$ is compactly supported with $\phi=1$ on $[-R,R]$, we get $\| h_0 (1-\phi) \| \leq \varepsilon$. On the compact support of such $\phi$ we can assume $|U|\leq \varepsilon$ since $u\to1$, which implies $\mathcal{N}(U) \leq C_A|U|^2 \leq \varepsilon|U|$. We therefore have
\begin{align*}
    & \|\mathcal{N}(U)\eta\| \leq \|\mathcal{N}(U)\eta \phi\| + \|\mathcal{N}(U)\eta(1-\phi)\| \leq \varepsilon \|U\eta \phi\| + C_A \|U\eta(1-\phi)\| \leq \\
    & \varepsilon\| U\eta \| + C_A \| b(\tau) h_0 (1-\phi)\| + C_A \| V(1-\phi) \| \leq \\
    & 2\varepsilon |b(\tau)| + C_A\| f\eta \| + C\exp(-cA^2 \tau)
\end{align*}

Putting these together, the ODE can be estimated as
\begin{align*}
    (b^2)_\tau &\geq 2b^2 -b^2 -\big(\|\mathcal{N}(U)\eta\|^2 + \| \mathcal{G}\eta \|^2 +\| \mathcal{E}\|^2 \big) \\
    & \geq b^2 - 4 \varepsilon^2 b^2 - C_A \| f\eta \|^2- \varepsilon^2 \|f\eta\|^2 - C\exp(-cA^2 \tau) \\
    & \geq (1-4\varepsilon^2)b^2 - C_A \|f \eta\|^2 -C\exp(-cA^2 \tau)
\end{align*}

If we have the Domination Condition \ref{onemodedoms}, the previous argument still works with the following slight modifications. This time we write $U\eta$ as
$$U\eta = b(\tau) h_0 + b_{m+1}(\tau) h_{m+1}+ W$$
and when we pick $R$, we pick it large enough so that $\phi$ additionally satisfies ${\|h_{m+1} (1-\phi) \| \leq \varepsilon}$. Note that we get the estimates $|b_{m+1}|\leq C \| f\eta \|$ and $$\| W \|^2 \leq C(\| f\eta \|^2 - a_m^2(\tau)) +C\exp(-cA^2 \tau)\leq \varepsilon \| f\eta \|^2 + C\exp(-cA^2 \tau)$$ 
\end{proof}

The previous Lemma allows us to bound $b$ by $\|f\eta\|$.

\begin{prop}\label{nocstforu}
Assume the Decay Condition \ref{atmostexpdecay} holds and $A$ is chosen large compared to $\Lambda$. For some constant $C'_A$, we eventually have $$b^2(\tau) \leq C'_A \|f\eta(\tau)\|^2$$
If the Domination Condition \ref{onemodedoms} holds, this can be improved to $b^2(\tau) \leq \varepsilon \|f\eta(\tau)\|^2$ for any $\varepsilon$ and large enough $\tau$.
\end{prop}
\begin{proof}
Since $A$ is large compared to $\Lambda$, we may drop the exponential terms and replace them with $\varepsilon \| f\eta\|$. Both statements are proven similarly. Let $C_A$ be as in Lemma \ref{bode}. Assume that for some sequence $\tau_i \to \infty$ we have $b^2(\tau_i) > 8C_A \|f\eta(\tau_i)\|^2$ when proving the first statement, or $b^2(\tau_i) > \varepsilon_0 \|f\eta(\tau_i)\|^2$ for some $\varepsilon_0$ for the second statement. Additionally, for proving the second statement, pick $\varepsilon\leq \frac{1}{2}$ in Lemma \ref{bode} small enough such that $\varepsilon^2 \frac{2}{\varepsilon_0} \leq \frac{1}{4}$, and choose $j$ large enough so that the conclusion of Lemma \ref{bode} holds for our choice of $\varepsilon$ and $\tau\geq \tau_j$. Now let $\tau_*$ be the supremum of all $\tau$'s such that on $[\tau_j, \tau]$ we have $b^2(\tau) \geq 4C_A \|f\eta(\tau)\|^2$ or $b^2(\tau) \geq \frac{\varepsilon_0}{2} \|f\eta(\tau)\|^2$. By Lemma \ref{bode} and our choice of $\tau_j$ and $\varepsilon$, on $[\tau_j, \tau_*]$ we have \begin{gather*}
    \frac{d}{d\tau}b^2 \geq (1-\varepsilon^2)b^2 - \frac{1}{4}b^2 \geq \frac{1}{2}b^2 \Rightarrow \\
    b^2(\tau) \geq e^{\frac{1}{2}(\tau-\tau_j)}b^2(\tau_j) > \gamma e^{\frac{1}{2}(\tau-\tau_j)} \| f\eta\|^2(\tau_j)
\end{gather*}
where $\gamma$ is $8C_A$ or $\varepsilon_0$. Whether the neutral mode of $f\eta$ dominates other modes or we have exponential decay, for every $\delta>0$ we have $\| f\eta\|(\tau) \leq e^{\delta(\tau-\tau_j)} \| f\eta\|(\tau_j)$. So $b^2(\tau) > \gamma \| f\eta\|^2(\tau)$. This shows that in fact we have $\tau_* = \infty$, which implies $\|U\eta\|$ is blowing up exponentially. This contradicts the boundedness of $U\eta$.
\end{proof}

Controlling the constant mode of $U\eta$ allows us to get pointwise control of $u-1$ near the origin, which will be needed later.
\begin{lem} \label{betterbaseptest}
Under the Decay Condition \ref{atmostexpdecay}, on any compact set $|\sigma| \leq R$ we have
$$|u-1| \leq C \| f\eta \|$$
\end{lem}
\begin{proof}
As is shown later in the proof of Lemma \ref{bode}, Proposition \ref{nocstforu} implies the bounds $\| U\eta \| \leq C \| f\eta \|$, which gives $|U| \leq C \| f\eta \|$ on any compact set by the Sobolev inequality. Since $u \to 1$ on any compact set and $\log u \approx u-1$, we have the same bound for $|u-1|$.
\end{proof}

Similar to the previous two Lemmas, we can show the following in case our Decay Condition fails. This establishes part of our main result, Theorem \ref{main}.
\begin{lem}
    Assume for every $A$, we eventually have
    $$\| f(\tau, \sigma) \chi(\frac{\sigma}{A\sqrt{\tau}}) \| \leq C_A \exp(-cA^2 \tau)$$
    Then on any set $\sigma \leq A \sqrt{\tau}$, we have $|u-1| \leq C \exp(-A\tau)$ for any $A$ and large enough $\tau$.
\end{lem}

\subsection{Proof of Theorem \ref{main}}
In this subsection we prove the main result, Theorem \ref{main}, assuming the contents of sections \ref{neutral} and \ref{exp}. The main theorem follows from the Proposition below by Sobolev embedding. Notice that the eigenvalues of $\mathcal{L}$ and $\mathcal{A}$ in the equations for $U$ and $f$ differ by $\frac{1}{2}$ and we are working with $\mathcal{A}$. Hence, unlike the usual convention, the neutral mode corresponds to $m=1$.
\begin{prop}
Consider the expansions for $U=\log u$ and $f=U_\sigma = \frac{u_\sigma}{u}$ in the orthonormal basis of eigenfunctions
\begin{gather*}
    U\eta = \sum_{k=0}^\infty b_k(\tau) h_k(\sigma),\quad
    f\eta = \sum_{k=0}^\infty a_k(\tau) h_k(\sigma).
\end{gather*}
There exists some $m \geq 1$ such that
\begin{gather*}
    \sum_{k\neq m}^\infty a_k^2(\tau) = o(a_m^2(\tau)),\quad
    \sum_{k\neq m+1}^\infty b_k^2(\tau) = o(b_{m+1}^2(\tau))
\end{gather*}
If $m=1$ then $b_2(\tau)= a_1(\tau) = \left(\frac{\sqrt[4]{\pi}}{2} + o(1) \right) \frac{1}{\tau}$.\newline
If $m \geq 2$, then ${b_{m+1}(\tau) = \sqrt{\frac{2}{m+1}} a_m(\tau) = C e^{-\frac{m-1}{2}\tau}}$.
\end{prop}
\begin{proof}
The first assertion is proven jointly in Proposition \ref{MZ} and Proposition \ref{onemodedoms} in Section \ref{exp}. The coefficient $a_m$ is computed in Section \ref{neutral} in the $m=1$ case, and in Proposition \ref{onemodedoms} in the $m\geq 2$ case. The relation between $a_m$ and $b_{m+1}$ follows from Lemma \ref{recurs2} and the Remark following it.
\end{proof}

\subsection{Proof of Proposition \ref{feqlocest}}

We now deal with Proposition \ref{feqlocest}. We estimate the terms in $F$ one by one. We assume $\tau_0$ is large enough so that results of section \ref{apriori} hold for $\tau \geq \tau_0$.
\begin{lem}
For $\tau\geq \tau_0$, universal $C$ and $N=N(g_0)$ we have
$$
\int \left( \int_0^\sigma f^2(\xi)\, d\xi  \right)^2 f_\sigma^2 \eta^2 \rho \, d\sigma \leq C N^2 \int f^4 \theta^4 \rho \, d\sigma
$$
\end{lem}
\begin{proof}
Using Holder's inequality and then changing the order of integration we have
\begin{align*}
    & \int \left( \int_0^\sigma f^2(\xi)\, d\xi  \right)^2 f_\sigma^2 \eta^2 \rho \, d\sigma \leq \int_{-\infty}^\infty |\sigma| \left| \int_0^\sigma f^4(\xi)\, d\xi  \right| f_\sigma^2 \eta^2 \rho \, d\sigma\\
    & = \left( \int_{-\infty}^0 \int_{-\infty}^\xi + \int_0^\infty \int_\xi^\infty \right) |\sigma| f^4(\xi)  f_\sigma^2(\sigma) \eta^2(\sigma) \rho(\sigma) \, d\sigma \, d\xi\\
    & \leq N^2\left( \int_{-\infty}^0 \int_{-\infty}^\xi + \int_0^\infty \int_\xi^\infty \right) |\sigma| f^4(\xi)  \eta^2(\xi) \rho(\sigma) \, d\sigma \, d\xi \\
    & \leq C N^2\left( \int_{-\infty}^0 + \int_0^\infty \right) f^4(\xi) \eta^2(\xi) \rho(\xi) \, d\xi \leq C N^2 \int_{-\infty}^\infty f^4(\xi) \theta^4(\xi) \rho(\xi)\, d\xi
\end{align*}
In the above calculations, we have used the following points: $f_\sigma$ is bounded by some $N$ on the support of $\eta$; since $r\chi'(r) \leq 0$ we have $\eta(\sigma) \leq \eta(\xi)$ for $|\xi|\leq |\sigma|$; we have the estimate
$\int_\xi^\infty |\sigma| \exp(\frac{-\sigma^2}{4})\, d\sigma \leq C \exp(\frac{-\xi^2}{4})$ for $\xi>0$ and a similar estimate for the integral from $-\infty$ to $\xi<0$; the specific choice of $\chi$ implies $\theta=1$ on the support of $\eta$.
\end{proof}

\begin{lem}
Assuming $u \geq \frac{1}{2}$ and $|f|\leq M$ on the support of $\eta$, for any $\varepsilon>0$ there exist $\tau_\varepsilon=\tau_\varepsilon(\tau_0, \varepsilon)$ and constants $C$ and $C_\varepsilon=C_\varepsilon(g_0, \varepsilon)$ such that for $\tau\geq \tau_\varepsilon$ we have
\begin{gather*}
\int \left( \frac{1}{u^2}-1 \right)^2 f^2 \eta^2 \rho \, d\sigma \leq \\
C \, \varepsilon \int f^2 \eta^2 \rho \, d\sigma + C_\varepsilon \int f^4 \theta^4 \rho \, d\sigma+ C M^4 \exp(- \frac{A^2}{8} \tau)
\end{gather*}
\end{lem}
\begin{proof}
Let us denote $u(\tau, 0)$ by $u_0$ and denote by $C_0$ any constant that depends only on $g_0$. For $\tau\geq \tau(g_0, \varepsilon)$ we can take $(\frac{1}{u_0^2}-1)^2 \leq \varepsilon$. We have
\begin{align*}
    & \frac{1}{u^2} = \frac{1}{u_0^2} + \int_0^\sigma (-2)\frac{f}{u^2}\, d\xi \Rightarrow (\frac{1}{u^2}-1)^2 \leq 2 (\frac{1}{u_0^2}-1)^2 + 8\left( \int_0^\sigma \frac{f}{u^2} \, d\xi \right)^2 \\
    & \leq 2 \varepsilon + 8 |\sigma| \left|\int_0^\sigma \frac{f^2}{u^4}\, d\xi \right| \leq 2 \varepsilon + C \left| \frac{\sigma}{2} \right| \left| \int_0^\sigma f^2(\xi)\, d\xi \right| \Rightarrow\\
    & \int \left( \frac{1}{u^2}-1 \right)^2 f^2 \eta^2 \rho \leq 2\varepsilon \int f^2 \eta^2 \rho + C\int \left|\frac{\sigma}{2}\right| \left|\int_0^\sigma f^2(\xi)\, d\xi \right| f^2 \eta^2 \rho \, d\sigma
\end{align*}
For estimating the second term, we integrate by parts by taking $dv=\frac{\sigma}{2}\rho\,d\sigma$. (Note that since $f^2$ is positive we can drop the absolute values)
\begin{align*}
    & \int \frac{\sigma}{2} \left(\int_0^\sigma f^2(\xi)\, d\xi \right) f^2 \eta^2 \rho \, d\sigma = \\
    & \int f^4 \eta^2 \rho + 2 \int \left(\int_0^\sigma f^2(\xi)\, d\xi \right) ff_\sigma \eta^2 \rho \, d\sigma + 2 \int \left(\int_0^\sigma f^2(\xi)\, d\xi \right) f^2 \eta \eta_\sigma \rho \, d\sigma\\
    & \leq \int f^4 \theta^4 \rho + 2 \left( \int f^2 \eta^2 \rho \right)^\frac{1}{2} \left( \int \left( \int_0^\sigma f^2(\xi)\, d\xi  \right)^2 f_\sigma^2 \eta^2 \rho \, d\sigma \right)^\frac{1}{2} \\
    & + 2 \left( \int f^4 \eta^2 \rho \right)^\frac{1}{2} \left( \int \left( \int_0^\sigma f^2(\xi)\, d\xi  \right)^2  \eta_\sigma^2 \rho \, d\sigma \right)^\frac{1}{2} \leq \\
    & \int f^4 \theta^4 \rho + \varepsilon \int f^2 \eta^2 \rho + \frac{CN^2}{\varepsilon} \int f^4 \theta^4 \rho + \int f^4 \theta^4 \rho + CM^4 \exp(-\frac{A^2}{8} \tau)
\end{align*}
where $C$ is an independent constants. In the last line we have used the previous lemma as well as
$$
\int \left( \int_0^\sigma f^2(\xi)\, d\xi  \right)^2  \eta_\sigma^2 \rho \, d\sigma \leq M^4 \int \sigma^2  \eta_\sigma^2 \rho \, d\sigma \leq CM^4 \int_{|\sigma|\geq A\sqrt{\tau}}\sigma^2e^\frac{-\sigma^2}{4}
$$
and this integral is less than $C \exp(- \frac{A^2}{8} \tau)$.
\end{proof}

\begin{lem}
For $\tau \geq \tau_0$ and some universal $C$ we have
\begin{gather*}
    \int f^6 \eta^2 \rho \leq M^2 \int f^4 \theta^4 \rho\\
    |E| \leq C(M+N)\\
    \| E \|^2(\tau) = \int E^2(\sigma, \tau) e^{-\frac{\sigma^2}{4}}\, d\sigma \leq C (M^2+N^2) \exp(- \frac{A^2}{8} \tau)
\end{gather*}
\end{lem}
\begin{proof}
We have
\begin{gather*}
    \eta_\tau= \chi'(\frac{\sigma}{A \sqrt{\tau}})(\frac{-\sigma}{2A \sqrt{\tau^3}})\\
    \eta_\sigma = \chi'(\frac{\sigma}{A \sqrt{\tau}}) \frac{1}{A\sqrt{\tau}}\qquad \eta_{\sigma \sigma}= \chi''(\frac{\sigma}{A \sqrt{\tau}}) \frac{1}{A^2\tau} 
\end{gather*}
which are all bounded by some constant. Hence $E$ is non-zero only for $A\sqrt{\tau} \leq |\sigma| \leq 2A \sqrt{\tau}$. Noting that $f, f_\sigma$ are bounded and the bounds for the derivatives of $\eta$ and the estimates
\begin{gather*}
\int_{|\sigma|\geq a} e^\frac{-\sigma^2}{4} \, d\sigma \leq C \exp(-\frac{a^2}{8  })\\
\int_{|\sigma|\geq a} \sigma^2 e^\frac{-\sigma^2}{4} \, d\sigma = \left. -2\sigma e^\frac{-\sigma^2}{4}\right \rvert_{\pm a}^{\pm \infty} + 2 \int_{|\sigma|\geq a} e^\frac{-\sigma^2}{4} \, d\sigma \leq C \exp(-\frac{a^2}{8})
\end{gather*}
for $a>1$ we get what we have claimed.
\end{proof}

Putting all the previous lemmas together, we conclude the proof of Proposition \ref{feqlocest}.

\subsection{Proof of Proposition \ref{idifineq}}

Now we prove the differential inequality for $I$ in Proposition \ref{idifineq}. From now on, we assume $\tau\geq \tau_0$ and 
\begin{gather*}
u\geq \frac{1}{2}\\
|f| \leq \frac{C_0}{\tau} \leq M \\
|f_\sigma|\leq N     
\end{gather*}
on the support of $\theta$ 

Replacing $\eta$ with $\theta$ in equation \ref{feqloc}, multiplying both sides of it by $(f\theta)^3 \sigma^k \rho$ and integrating gives
\begin{gather} \label{ieq1}
    \frac{1}{2}I I_\tau= \int \mathcal{A}(f\theta) (f\theta)^3 \sigma^k \rho -n \int \left( \int_0^\sigma f^2(\xi) \, d\xi \right) f^3 f_\sigma \theta^4 \sigma^k \rho \\
    + \int \left( \frac{1}{u^2}-1 \right)(f\theta)^4 \sigma^k \rho - n\int f^6 \theta^4 \sigma^k \rho + \langle E,(f\theta)^3\sigma^k \rangle \nonumber
\end{gather}
where $\langle \, , \rangle$ is the $L^2(\rho\, d\sigma)$ inner product. The nontrivial part is controlling the first and second term in the equation for $2II_\tau$.

\begin{lem}
\begin{gather*}
    \int \mathcal{A}(f\theta) (f\theta)^3 \sigma^k \rho = \\
    \frac{1}{2}I^2 - 3 \int (f\theta)^2 (\partial_\sigma (f \theta))^2 \sigma^k \rho + \frac{k(k-1)}{4}\int (f\theta)^4 \sigma^{k-2}\rho - \frac{k}{8} I^2
\end{gather*}
\end{lem}
\begin{proof}
\begin{align*}
    & \int \mathcal{A}(f\theta) (f\theta)^3 \sigma^k \rho = \int \left(\partial_\sigma(\rho \partial_\sigma (f\theta) ) + \frac{1}{2} f\theta \rho \right) (f\theta)^3 \sigma^k = \\
    & - 3 \int (f\theta)^2 (\partial_\sigma (f \theta))^2 \sigma^k \rho - k \int (\partial_\sigma (f \theta))(f\theta)^3 \sigma^{k-1} \rho + \frac{1}{2}I^2
\end{align*}
For the middle term, we integrate by parts again
\begin{align*}
    & - k \int (\partial_\sigma (f \theta))(f\theta)^3 \sigma^{k-1} \rho = - \frac{k}{4} \int \partial_\sigma( (f \theta)^4) \sigma^{k-1} \rho \\
    & = \frac{k(k-1)}{4} \int (f\theta)^4 \sigma^{k-2} - \frac{k}{8} \int f^4 \theta^4 \sigma^k \rho
\end{align*}
\end{proof}

We now define $P$ as
$$
P^2=\int f^4 \theta^4 \sigma^{k-2} \rho
$$
\begin{lem} \label{P}
For any positive $\varepsilon$ and $\delta$, there exists $\tau_{\varepsilon \delta}$ depending on $\varepsilon$ and $\delta$ such that for $\tau\geq \tau_{\varepsilon \delta}$ we have
$$ P^2 =  \int f^4 \theta^4 \sigma^{k-2} \rho \leq \varepsilon \delta^{2-\frac{k}{2}} I \| f\theta \|+ \delta^2 I^2$$
\end{lem}
\begin{proof}
We can choose $\tau_{\varepsilon\delta}$ large enough so that for $\tau\geq \tau_{\varepsilon\delta}$, $|\sigma|\leq \delta^{-1}$ implies $|f|\leq \varepsilon$. By Holder inequality we get
\begin{align*}
    & P^2 =  \int f^4 \theta^4 \sigma^{k-2} \rho \leq \left( \int f^4 \theta^4 \sigma^k \rho \right)^\frac{1}{2}\left(\int f^4 \theta^4 \sigma^{k-4} \rho \right)^\frac{1}{2} \leq \\
    & I\left(\int_{|\sigma| \leq \delta^{-1}} f^4 \theta^4 \sigma^{k-4} \rho \right)^\frac{1}{2}+ I \left(\int_{|\sigma| > \delta^{-1}} f^4 \theta^4 \sigma^{k-4} \rho \right)^\frac{1}{2} \leq\\
    &\delta^{2-\frac{k}{2}} I \left(\int_{|\sigma| \leq \delta^{-1}} f^4 \theta^4  \rho \right)^\frac{1}{2}+ \delta^2 I \left(\int f^4 \theta^4 \sigma^k \rho \right)^\frac{1}{2} \leq \varepsilon \delta^{2-\frac{k}{2}} I \| f\theta \|+ \delta^2 I^2
\end{align*}
\end{proof}

\begin{lem}
For $\tau\geq \tau_0$ as in the beginning of this part and any even number $k$, we have
\begin{gather*}
    \int \left( \int_0^\sigma f^2(\xi) \, d\xi \right) f^3 f_\sigma \theta^4 \sigma^k \rho \, d\sigma \geq\\
    - \frac{1}{4} \int  f^6 \theta^4 \sigma^k \rho -M^4 C_k \, I \exp(-\frac{1}{2}^2\tau) - \frac{C_A}{\tau} C_k I^2
\end{gather*}
where $C_A$ (from Lemma \ref{basicderestim}) depends only on $g_0$ and $A$ and $C_k$ depends only on $k$.
\end{lem}
\begin{proof}
We integrate by parts taking $dv= f^3 f_\sigma \, d\sigma$.
\begin{align} \label{eq1}
    & \int \left( \int_0^\sigma f^2(\xi) \, d\xi \right) f^3 f_\sigma \theta^4 \sigma^k \rho \, d\sigma = \nonumber\\
    & - \frac{1}{4} \int  f^2 f^4 \theta^4 \sigma^k \rho -\int \left( \int_0^\sigma f^2(\xi) \, d\xi \right) f^4 \theta^3 \theta_\sigma \sigma^k \rho \nonumber \\
    & - \frac{k}{4}\int \left( \int_0^\sigma f^2(\xi) \, d\xi \right) f^4 \theta^4 \sigma^{k-1} \rho + \frac{1}{8}\int \left( \int_0^\sigma f^2(\xi) \, d\xi \right) f^4 \theta^4 \sigma^{k+1} \rho
\end{align}
We deal with the second, third and fourth term in the right hand side of \ref{eq1} above in order. For the second term we use Holder inequality
\begin{align*}
    & \int \left( \int_0^\sigma f^2(\xi) \, d\xi \right) f^4 \theta^3 \theta_\sigma \sigma^k \rho \leq\\
    & \left( \int f^4 \theta^4 \sigma^k \rho \right)^\frac{1}{2}\left( \int \left( \int_0^\sigma f^2(\xi) \, d\xi \right)^2 f^4 \theta^2 \theta_\sigma^2 \sigma^k \rho \right)^\frac{1}{2} \leq I \left( \int M^8 \theta^2 \theta_\sigma^2 \sigma^{k+2} \rho \right)^\frac{1}{2} \leq \\
    & M^4 C_k \, I \exp(-\frac{1}{2}A^2\tau)
\end{align*}
for some constant $C_k$ depending only on $k$. In the last equality we have used the fact that $\theta_\sigma \neq 0$ only for $|\sigma|> 2A \sqrt{\tau}$ so
$$ \int \theta_\sigma^2 \sigma^{k+2} \rho \leq \exp(-\frac{4A^2\tau}{8}) \int \sigma^{k+2} \exp(-\frac{\sigma^2}{8}) \, d\sigma$$

For the third term of \ref{eq1} we use $|f|\leq \frac{C_A}{\sqrt{\tau}}$. We have
\begin{align*}
    & \sigma \int_0^\sigma f^2(\xi) \, d\xi \leq \sigma \int_0^\sigma \frac{C_A}{\tau} = C_A\frac{\sigma^2}{\tau} \Rightarrow\\
    & -\frac{k}{2} \int \left( \int_A^\sigma f^2(\xi) \, d\xi \right) f^4 \theta^4 \sigma^{k-1} \rho \geq -\frac{k}{2}\int C_A \frac{\sigma^2}{\tau} f^4 \theta^4 \sigma^{k-2} \rho = -\frac{k}{2}\frac{C_A}{\tau} I^2 
\end{align*}

The fourth term in \ref{eq1} is positive since $k$ is even and $\sigma \int_0^\sigma f^2(\xi) \, d\xi$ is positive.
\end{proof}

We can now finish the proof of Proposition \ref{idifineq}.
\begin{proof}[Proof (of Proposition \ref{idifineq})]
First observe that
\begin{align*}
    & \langle E,(f\theta)^3\sigma^k \rangle = \int E(\sigma) (f\theta)^3\sigma^k \rho \, d\sigma \leq \left( \int E^4 \sigma^k \rho \, d\sigma \right)^\frac{1}{4} \left( \int (f\theta)^4 \sigma^k \rho \, d\sigma \right)^\frac{3}{4} \leq \\
    & \left( \int C(M+N)^4 \sigma^k \exp(-\frac{A^2 \tau}{8})\exp(-\frac{\sigma^2}{8})   \, d\sigma \right)^\frac{1}{4} I^\frac{3}{2} \leq \\
    & C (M+N) C_k \exp(-\frac{1}{32}A^2 \tau) I^\frac{3}{2}
\end{align*}
where we have used the obvious pointwise bound for $E$ and the fact that $E$ is only non-zero for $|\sigma|>A\sqrt{\tau}$. Also note $\frac{1}{u^2}-1 \leq C$ on the support of $\theta$.

We put the previous estimates and \ref{ieq1} together. For any $\varepsilon$ and $\delta$ and large enough $\tau$ we get
\begin{align*}
    & \frac{1}{2}II_\tau\leq \\
    & \frac{1}{2}I^2 - 3 \int (f\theta)^2 (\partial_\sigma (f \theta))^2 \sigma^k \rho + \frac{k(k-1)}{4}\int (f\theta)^4 \sigma^{k-2}\rho - \frac{k}{8} I^2 \\
    & + \frac{n}{4} \int  f^6 \theta^4 \sigma^k \rho +n M^4 C_k \, I \exp(-\frac{1}{2}A^2\tau) + \frac{C_A}{\tau} C_k I^2\\
    & +C I^2 - n\int f^6 \theta^4 \sigma^k \rho + C (M+N) C_k \exp(-\frac{1}{32}A^2 \tau) I^\frac{3}{4} \leq \\
    & (\frac{1}{2}-\frac{k}{8}+\frac{C_A}{\tau} C_k+C)I^2 + \frac{k(k-1)}{4}(\varepsilon \delta^{2-\frac{k}{2}} I \| f\theta \|+ \delta^2 I^2) \\
    &  + n M^4 C_k \, I \exp(-\frac{1}{2}A^2\tau)+ C (M+N) C_k \exp(-\frac{1}{32}A^2 \tau) I^\frac{3}{2}
\end{align*}
Cancelling the $\frac{1}{2}I$ and using Cauchy inequality on the term with $I^\frac{1}{2}$ we get
\begin{align*}
    & I_\tau \leq (2-\frac{k}{4}+\frac{2 C_A}{\tau} C_k+2C + \frac{k(k-1)}{2} \delta^2)I+\frac{k(k-1)}{2} \varepsilon \delta^{2-\frac{k}{2}} \| f\theta \|\\
    & + 2 n M^4 C_k \, \exp(-\frac{1}{2}A^2\tau)+ C^2 (M+N)^2 C_k^2 \exp(-\frac{1}{16}A^2 \tau)
\end{align*}

Now we proceed as follows: First we pick $k$ even and large enough so that $2 + 2C_\ell \leq \frac{k}{24}$ and $\frac{k}{8}>a$. Then we pick $\delta$ small so $\frac{k(k-1)}{2} \delta^2 \leq \frac{k}{24}$ and $\varepsilon$ such that $\frac{k(k-1)}{2}\varepsilon \delta^{2-\frac{k}{2}} \leq \varepsilon_1$. This gives us $\tau_{\varepsilon \delta}$ of Lemma \ref{P}. Now we pick $\tau_2 \geq \tau_0 , \tau_{\varepsilon\delta}$ so large (depending on $k$ and $A$) that $C_k \exp(- \frac{1}{2} A^2\tau)\, , \, C_k^2 \,\exp(-\frac{1}{16}A^2\tau) \leq \exp(-\frac{1}{32}A^2\tau)$ and $\frac{2C_A}{\tau} C_k \leq \frac{k}{24}$ for $\tau\geq \tau_2$.
\end{proof}

\subsection{Proofs of Proposition \ref{syswexp}}

We only need to put Propositions \ref{feqlocest} and \ref{idifineq} together. We have the differential inequalities
\begin{gather*}
    \frac{dx}{d\tau}  \geq \frac{1}{2}x - C \| F \| - C\|E\| \\
    |\frac{dy}{d\tau} | \leq  C \| F \| + C\|E\| \\
    \frac{dz}{d\tau}  \leq -\frac{1}{2}z + C \| F \| + C\|E\|
\end{gather*}
for some $C$, where for $c=\frac{1}{16}$ and $\tau\geq \tau_0$, $F$ and $E$ satisfy the estimates
\begin{gather*}
    \|F\| \leq C_0 \varepsilon \| f \eta \|+ C_\varepsilon \| f^2 \theta^2\| + C_0 \exp(- c A^2 \tau)\nonumber \\
    \|E\| \leq C_0 \exp(- c A^2 \tau)
\end{gather*}

Assume $k$ is known in the definition of $I$ and pick any $\varepsilon_1>0$. We can pick a small enough $\varepsilon$ in Proposition \ref{feqlocest} and then small enough $\varepsilon$ and $\delta$ in inequality \ref{quadleqi} to get $\tau_1(\tau_0, \varepsilon_1)$, such that for $\tau \geq \tau_1$ we have
\begin{gather*}
    C \| F \| + C\|E\| \leq  \varepsilon_1(x+y+z) + \varepsilon_1\| f\theta \| + \varepsilon_1 I+ C_0 \exp(-cA^2 \tau)
\end{gather*}

Now note that for $\tau \geq \tau_0$
$$
|\| f\theta \|^2 - \| f\eta \|^2 | \leq \left | \int_{A\sqrt{\tau} \leq |\sigma| \leq 2A\sqrt{\tau}} f^2 e^{-\frac{\sigma^2}{4}}\, d\sigma \right| \leq C M^2 \exp(-cA^2 \tau)
$$
So we can replace the $\|f\theta\|$ term with $\|f\eta\| \leq x+y+z$ and get the system
\begin{gather*}
    \frac{dx}{d\tau}  \geq \frac{1}{2}x - 2\varepsilon_1(x+y+z) - \varepsilon_1 I- C_0 \exp(-cA^2 \tau) \\
    |\frac{dy}{d\tau} | \leq  2\varepsilon_1(x+y+z) + \varepsilon_1 I+ C_0 \exp(-cA^2 \tau) \\
    \frac{dz}{d\tau} \leq -\frac{1}{2}z + 2\varepsilon_1(x+y+z) + \varepsilon_1 I + C_0 \exp(-cA^2 \tau)
\end{gather*}

By picking $a> \frac{1}{2}$ in Proposition \ref{idifineq}, we have
\begin{gather*}
    \frac{d}{d\tau} I \leq -\frac{1}{2} I + \varepsilon_1 \| f\theta \| + C_0\exp(-cA^2\tau)
\end{gather*}
for $\tau \geq \tau_2(\tau_0, k, A, \varepsilon_1)$. Letting $\zeta(\tau) = z(\tau)+ I(\tau)$ we get
\begin{gather*}
    \frac{dx}{d\tau}  \geq \frac{1}{2}x - 3\varepsilon_1(x+y+\zeta) - C_0 \exp(-cA^2 \tau) \\
    |\frac{dy}{d\tau} | \leq  3\varepsilon_1(x+y+\zeta) + C_0 \exp(-cA^2 \tau) \\
    \frac{d\zeta}{d\tau} \leq -\frac{1}{2}\zeta + 3\varepsilon_1(x+y+\zeta) + C_0 \exp(-cA^2 \tau)
\end{gather*}
for some universal $c$ and $\tau$ sufficiently large. Since $\varepsilon_1$ was arbitrary, we are done.

\section{The Neutral Case}\label{neutral}

In this section we handle the case $x+\zeta = o(y)$ in Proposition \ref{MZ} and derive the asymptotic ODE governing the coefficient of neutral mode of $f$. In particular, we are assuming
\begin{gather} \label{atmostexpdecayrem} 
\| f\eta \| \geq C_0 \exp(-\Lambda \tau) \gg C\exp(-cA^2\tau)
\end{gather}

Recall that we have set $U=\log u$, and we have the following equations for $U\eta$ and $f\eta$.
\begin{gather*}
    (U\eta)_\tau = \mathcal{L}(U\eta) + \mathcal{N}(U)\eta - \mathcal{G}\eta + \mathcal{E}
\end{gather*}
\begin{gather*}
    (f\eta)_\tau= \mathcal{A}(f\eta) -n\left( \int_0^\sigma f^2(\xi)\, d\xi  \right) f_\sigma \eta+ \left( \frac{1}{u^2}-1 \right)f\eta -nf^3\eta + E
\end{gather*}
The quadratic term in the ODE for the dominant mode of $f$ comes from $(\frac{1}{u^2}-1)f$, while the other terms are essentially cubic. We will show that the inner product of these cubic terms and $h_1$, as well as the part of the quadratic term coming from non-dominant modes of $f$ and $u$, are in fact negligible in the ODE. Controlling the quadratic term here is the main reason for the second stronger assertion in Proposition \ref{nocstforu}. Since one mode of $f\eta$ is dominating the rest, the unstable mode of $U=\log u \approx \frac{-1}{2}(\frac{1}{u^2}-1)$ corresponding to eigenvalue $1$ gets arbitrary small compared to $f$ by Proposition \ref{nocstforu}.

We need some general results about projections of related functions on eigenspaces $\mathfrak{H}_m$. Recall that the unit eigenfunctions $h_m$ spanning $\mathfrak{H}_m$ are related to standard Hermite polynomials $H_m$ through $h_m (\sigma) = c_m H_m(\frac{\sigma}{2})$ where $c_m = \left(\frac{1}{2^m \sqrt{4\pi} m!}\right)^\frac{1}{2}$. They satisfy the recursive relations
\begin{gather}
    h_{m+1} = \frac{1}{\sqrt{2m+2}}(\sigma h_{m}-2h_{m}')\\
    h_{m+1}' = \sqrt{\frac{m+1}{2}} h_{m}
\end{gather}
 
\begin{lem}\label{recurs2}
For any function $F$ with appropriate smoothness and decay we have
\begin{gather}
    \langle F_\sigma \, , \, h_{m-1} \rangle = \sqrt{\frac{m}{2}} \langle F \, , \, h_{m} \rangle\\
    \langle \sigma F \, , \, h_{m} \rangle = 2\langle F_\sigma \, , \, h_{m} \rangle + \sqrt{2m}\langle F \, , \, h_{m-1} \rangle\\
    \langle \int_0^\sigma F(\xi)\, d\xi \, , \, h_{0} \rangle \leq C \left( \|F\|+\|F_\sigma\| \right)
\end{gather}
with the convention $h_{-1}=0$.
\end{lem}
We use this lemma to compare the modes of $U\eta$, $f\eta$ and the non-local term $\mathcal{G}$. Before continuing, we emphasize the following general remark to make the computations more tractable.
\begin{rem}
Suppose a function $F(\tau, \sigma)$ and its partial derivative $F_\sigma$ are bounded (or even have polynomial growth in $\sigma$), and take $\eta=\chi\left( \frac{\sigma}{A\sqrt{\tau}} \right)$ and $\theta=\chi\left( \frac{\sigma}{2A\sqrt{\tau}} \right)$ as before. Let $\mathcal{F} = F\theta$. We have the estimates
\begin{gather*}
    \| \mathcal{F} - F \mathds{1}_{supp \, \eta}\|^2 \leq C \int_{|\sigma|\geq A \sqrt{\tau}} F e^{\frac{-\sigma^2}{4}} \leq C \exp(-cA^2\tau)\\
    \| F_\sigma \theta - F_\sigma \mathds{1}_{supp \, \eta}\|^2 \leq C \int_{|\sigma|\geq A \sqrt{\tau}} F_\sigma e^{\frac{-\sigma^2}{4}} \leq C \exp(-cA^2\tau)\\
    \| \mathcal{F}_\sigma - F_\sigma \mathds{1}_{supp \, \eta}\|^2 \leq C \int_{|\sigma|\geq A \sqrt{\tau}} \big( F+F_\sigma \big) e^{\frac{-\sigma^2}{4}} \leq C \exp(-cA^2\tau)
\end{gather*}
These estimates show that up to an exponentially small error, we can work with $F$ and $F_\sigma$ instead of their localizations if we are only working on the support of $\eta$. All the functions that we apply Lemma \ref{recurs2} to are bounded. Since we have chosen $A$ large compared to $\Lambda$ as in \ref{atmostexpdecayrem}, we do not need to worry about such errors.
\end{rem}

Throughout this section, $a(\tau)$ denotes the neutral coefficient of $f\eta$:
\begin{gather*}
    a(\tau)= \langle f\eta(\tau) \, , \, h_1 \rangle = \int (f\eta)(\tau, \sigma) h_1(\sigma) e^{-\frac{\sigma^2}{4}}\, d\sigma
\end{gather*}
The ODE satisfied by $a(\tau)$ is
\begin{gather}
    \frac{d}{d\tau} a = \left\langle (\frac{1}{u^2}-1)f\eta \, , \, h_1 \right \rangle -n \left \langle \left( \int_0^\sigma f^2(\xi)\, d\xi  \right)f_\sigma \eta +f^3\eta \, , \, h_1 \right \rangle + \left\langle E \,, \, h_1 \right\rangle
\end{gather}
We deal with different terms one by one.

First, we deal with the quadratic term $(\frac{1}{u^2}-1)f\eta = (e^{-2U}-1)f\eta$ and compute the leading term of the ODE.

\begin{prop}\label{quad}
For any $\varepsilon$ and large enough $\tau$ we have
\begin{gather*}
    \left\langle (e^{-2U}-1)f\eta\, , \, h_1 \right\rangle = -\frac{2}{\sqrt[4]{\pi}}a^2(\tau) + \epsilon(\tau)
\end{gather*}
where $|\epsilon(\tau)| \leq \varepsilon a^2(\tau)$.
\end{prop}
\begin{proof}
Recall that $h_1= c_1 \sigma = \frac{1}{2\sqrt[4]{\pi}} \sigma$, and $e^{-2U}-1=-2U-2\mathcal{N}$. As before, since $u$ is bounded and away from zero, $U$ is bounded and we have
\begin{gather*}
    |\mathcal{N}|\leq CU^2\\
    |\mathcal{N}_\sigma| = |\frac{\partial}{\partial \sigma}(e^{-2U}-1+2U)| = |-2fe^{-2U}+2f| \leq C|f|\,|U|
\end{gather*}
for some $C$ (depending on $A$).
We have
\begin{gather*}
    \left\langle (e^{-2U}-1)f\eta\, , \, \sigma \right\rangle = -2\left\langle (U+\mathcal{N})f\eta\, , \, \sigma \right\rangle = -2\left\langle (U+\mathcal{N}) \sigma\, , \, f\eta \right\rangle
\end{gather*}

Using Lemma \ref{recurs2} we have
\begin{align}
    & \left\langle \sigma U\, , \, f\eta \right\rangle = \sum_{m=0} \left\langle \sigma U\, , \, h_m \right\rangle \left\langle f\eta \, , \,h_m \right\rangle \nonumber \\
    & = \sum_{m=0} \big( 2 \left\langle f \, , \, h_m\right\rangle + \sqrt{2m} \left\langle U \, , \,h_{m-1} \right\rangle \big) \left\langle f\eta \, , \, h_m \right\rangle \nonumber \\
    & = 2 \left\langle f \, , \, f\eta \right\rangle + \sqrt{2} \left\langle U \, , \, h_0\right\rangle \left\langle f\eta\, , \, h_1\right\rangle+ \sum_{m=2} \sqrt{\frac{4m}{m-1}}\left\langle f\, , \, h_{m-2}\right\rangle \left\langle f\eta\, , \, h_m\right\rangle
\end{align}
Now $|\left\langle U\, , \, h_0\right\rangle| \leq \varepsilon |a(\tau)|$ by Proposition \ref{nocstforu}. Also, if we define $g= \sum_{m=2} \left\langle f\, , \, h_{m-2}\right\rangle h_m$ and $\Tilde{f} = \sum_{m=2} \left\langle f\eta\, , \, h_m\right\rangle h_m$, we have $\| g \|\leq \| f\|$ and $\| \Tilde{f}\| \leq \varepsilon |a(\tau)|$. Therefore
$$ \left| \sum_{m=2} \sqrt{\frac{4m}{m-1}}\left\langle f\, , \, h_{m-2}\right\rangle \right| \leq C |\langle g\, , \, \Tilde{f} \rangle| \leq C \| g\| \| \Tilde{f}\| \leq \varepsilon a^2(\tau)$$
So the only significant term here is the first one, which is $2 \| f\eta\|^2$ up to an exponentially small error. This gives the leading term $$-4c_2 a^2= -\frac{2}{\sqrt[4]{\pi}} a^2$$

Note that as in the proof of Lemma \ref{bode}, we have $\| U\eta \| \leq C \| f\eta \|$. Indeed, disregarding exponentially small terms, $\| U\eta \|^2$ can be estimated as
$$|\langle U\eta \, , \, h_0 \rangle|^2+ \sum_{m \geq 1} |\langle U\eta \, , \, h_m \rangle|^2 \leq \varepsilon a(\tau)^2+ C \sum_{m \geq 0} |\langle f\eta \, , \, h_m \rangle|^2 \leq C \|f\eta\|^2$$
Now for $\varepsilon$, pick $R$ large enough so that for any compactly supported smooth function $\phi$ with $0\leq \phi \leq 1$ and $\phi=1$ on $|\sigma|\leq R$, we have $\| h_1(1-\phi^2)\| \leq \varepsilon$. If we write $f\eta = a(\tau)h_1+V$ with $\| V \|\leq \varepsilon |a(\tau)|$ we get
\begin{gather*}
    \| f\eta (1-\phi^2) \| \leq |a|\| h_1(1-\phi^2) \| + C\| V \| \leq 2\varepsilon |a(\tau)|
\end{gather*}
For such compactly supported $\phi$, we have
\begin{align*}
    & \big|\langle \sigma \mathcal{N}\, , \, f\eta \rangle \big| \leq \big|\langle \sigma \mathcal{N} \phi\, , \, f\eta\phi \rangle \big| +\big| \langle \sigma \mathcal{N}\, , \, f\eta(1-\phi^2) \rangle \big|
\end{align*}
On the compact support of $\phi$, we can assume $|\mathcal{N}| \leq C|U|^2\leq \varepsilon |U|$ so $\| \sigma \mathcal{N} \phi\| \leq \varepsilon \| U\eta\| \leq \varepsilon |a(\tau)|$. On the other hand, we have $\|f\eta(1-\phi^2) \| \leq 2\varepsilon |a(\tau)|$ and  $|\mathcal{N}|\leq C|U|$ and $|\mathcal{N}_\sigma| \leq C|f|$ since $U$ is bounded. Therefore
\begin{align*}
    & \| \sigma \mathcal{N} \| \leq C \big( \|\mathcal{N} \| + \|\mathcal{N}_\sigma \| \big) \leq C \big( \|U\| + \|f \| \big) \leq C |a(\tau)|
\end{align*}
Putting these estimates together we see 
\begin{gather}
    \big|\langle \sigma \mathcal{N}\, , \, f\eta \rangle \big| \leq \varepsilon a^2(\tau)
\end{gather}
and we are done.
\end{proof}

Next, we deal with the middle term. Note that since
\begin{gather*}
    \frac{\partial}{\partial \sigma} \left( \big( \int_0^\sigma f^2 \big) f  \right) = \big( \int_0^\sigma f^2\big)f_\sigma + f^3 
\end{gather*}
we can integrate by parts and rewrite this term as
\begin{gather*}
    n \left\langle \big( \int_0^\sigma f^2 \big) f \eta_\sigma \, , \, h_1 \right\rangle + n \left\langle \big( \int_0^\sigma f^2 \big) f \eta \, , \, \partial_\sigma h_1 - \frac{\sigma}{2}h_1 \right\rangle
\end{gather*}
But $\eta_\sigma$ is nonzero only for $|\sigma|\geq A \sqrt{\tau}$ and we have
\begin{gather*}
    \frac{\sigma}{2}h_1 - \frac{\partial}{\partial \sigma} h_1 = \frac{1}{4}h_2
\end{gather*}
So this is equal to
\begin{gather*}
    \frac{-n}{4} \left\langle \big( \int_0^\sigma f^2 \big) f \eta \, , \, h_2 \right\rangle + O(\exp(-cA^2 \tau))
\end{gather*}
for some universal $c$ since $|f|$ is bounded. We have the following estimate.

\begin{prop} \label{cubic}
For any $\varepsilon$, for large enough $\tau$ we have
\begin{gather}
    \left| \langle \big( \int_0^\sigma f^2 \big) f \eta \, , \, h_2\rangle \right| \leq \varepsilon a^2(\tau)
\end{gather}
\end{prop}

\begin{proof}
We can write
\begin{gather*}
    \left\langle \big( \int_0^\sigma f^2 \big) f \eta \, , \, h_2\right\rangle = \left\langle h_2 \big( \int_0^\sigma f^2 \big)\, , \, f\eta\right\rangle \\
    = \sum_{m=0}^\infty \left\langle c_2(\sigma^2-2) \big( \int_0^\sigma f^2 \big)\, , \, h_m\right\rangle \left\langle f\eta\, , \, h_m\right\rangle
\end{gather*}
Let
\begin{gather*}
    A_m = \left\langle (\sigma^2-2) \int_0^\sigma f^2\, , \, h_m \right\rangle
\end{gather*}
We use Lemma \ref{recurs2}. Adopting $h_{-1}=h_{-2}=0$ we can write
\begin{align*}
    & A_m= \left\langle \sigma^2 \int_0^\sigma f^2\, , \, h_m \right\rangle -2 \left\langle \int_0^\sigma f^2\, , \, h_m \right\rangle\\
    & = 2 \left \langle \frac{\partial}{\partial \sigma}\left( \sigma \int_0^\sigma f^2 \right)\, , h_m \, \right \rangle + \sqrt{2m} \left \langle \sigma \int_0^\sigma f^2 \, , \, h_{m-1} \right \rangle - 2 \left\langle \int_0^\sigma f^2\, , \, h_m \right\rangle\\
    & = 2\left\langle \sigma f^2\, , \, h_m \right\rangle + 2\sqrt{2m} \left\langle f^2\, , \, h_{m-1} \right\rangle + \sqrt{4m(m-1)}\left\langle \int_0^\sigma f^2\, , \, h_{m-2} \right \rangle
\end{align*}

If $m > 4$ we can further simplify this to
\begin{gather*}
    2\left\langle 2f f_\sigma\, , \, h_m \right\rangle + 4\sqrt{2m} \left\langle f^2\, , \, h_{m-1} \right\rangle + \sqrt{\frac{8m(m-1)}{m-2}}\left\langle f^2\, , \, h_{m-3} \right \rangle=\\
    4\left\langle f f_\sigma\, , \, h_m \right\rangle + 4\sqrt{\frac{4m}{m-1}} \left\langle 2f f_\sigma\, , \, h_{m-2} \right\rangle + \sqrt{\frac{16m(m-1)}{(m-2)(m-3)}}\left\langle 2ff_\sigma\, , \, h_{m-4} \right \rangle
\end{gather*}
Define $g=\sum_{m\geq 4} A_m h_m$ and $\Tilde{f}=\sum_{m\geq 4}  \left\langle f\eta\, , \, h_m \right\rangle h_m$. We have 
\begin{gather*}
\| g\| \leq C \| f f_\sigma\| \leq C \| f \| \\
\| \Tilde{f} \| \leq \varepsilon \| f\eta \|
\end{gather*}
So
\begin{gather}
    \left| \sum_{m\geq 4} A_m \left\langle f\eta, h_m \right\rangle \right| = |\langle g \, , \, \Tilde{f} \rangle| \leq \|g\| \| \Tilde{f} \| \leq C \varepsilon \| f \| \|f\eta\| \leq \varepsilon a^2(\tau)
\end{gather}

For $m \leq 3$, note that $\sqrt{2m}$ is bounded above and that $\langle \sigma f^2 \, , \, h_m \rangle = \langle f^2 \, , \, \sigma h_m \rangle$ where $\sigma h_m$ can be written as a linear combination of finitely many $h_i$s with universally bounded coefficients. So we get
\begin{gather*}
    |2\left\langle \sigma f^2\, , \, h_m \right\rangle + 2\sqrt{2m} \left\langle f^2\, , \, h_{m-1} \right\rangle| |\langle f\eta\, , \, h_m \rangle| \leq C \| f^2 \| \| f\eta \| \leq \varepsilon \| f\eta\|^2
\end{gather*}
So we are left with the $\langle \int_0^\sigma f^2\, , \, h_{m-2} \rangle$ term for $m =2, 3$. We have
\begin{align*}
    &\left|\left\langle \int_0^\sigma f^2\, , \, h_1 \right \rangle \langle f\eta\, , \, h_3 \rangle \right| = \sqrt{2} |\langle f^2\, , \, h_0 \rangle| |\langle f\eta\, , \, h_3 \rangle| \leq C \| f^2 \| \| f\eta \| \leq \varepsilon \| f\eta\|^2 \\
    & \left|\left\langle \int_0^\sigma f^2\, , \, h_0 \right \rangle \langle f\eta\, , \, h_2 \rangle \right| \leq C \big( \|f^2\|+ \|ff_\sigma\| \big) \varepsilon a(\tau) \leq C \| f \| \varepsilon a(\tau) \leq \varepsilon a^2(\tau)
\end{align*}
\end{proof}

For the error term, since $|f|$ and $|f_\sigma|$ are bounded we have
$$|\langle E, h_1\rangle| \leq \|E\| \leq C\exp(-cA^2 \tau)$$
as before. Putting Propositions \ref{cubic} and \ref{quad} together we have the asymptotic ODE
\begin{gather*}
    \frac{d}{d\tau} a = -\frac{2}{\sqrt[4]{\pi}}a^2 + o(1)a^2
\end{gather*}
which, together with $a(\tau) \to 0$ as $\tau \to \infty$, gives
\begin{gather}
    a(\tau) = \frac{\sqrt[4]{\pi}}{2\tau} + o(\frac{1}{\tau})
\end{gather}
Note that the Domination Condition \ref{onemodedoms} holds by assumption. Together with Proposition \ref{nocstforu} and $\langle U , h_2\rangle = \langle f , h_1\rangle $, the expression for $a(\tau)$ implies case 2 in our main Theorem.

\begin{rem}
Note that our normalization for $h_2$ is different from that of \cite{precise} but the final result is the same.
\end{rem}

\section{The Exponential Case} \label{exp}

In this section we deal with the case of exponential decay in Proposition \ref{MZ}, and establish the third case in our main Theorem. As in Proposition \ref{MZ2}, we let $\lambda= \lambda_m$ for $m \geq 2$ be the supremum of all exponential decay rates. We prove that in this case, the $m$-th mode of $f\eta$ dominates the other modes, and derive the ODE for the coefficient of $m$-th mode. From now on, by virtue of Proposition \ref{MZ2}, we may assume
\begin{gather*}
    \| f\eta \| + I \leq C \exp(-\beta \tau)
\end{gather*}
for any $\beta < \lambda =\lambda_m = \frac{m-1}{2}$ and large enough $\tau$. We may also assume the scale $A$ is the cut-off $\eta$ is taken large enough so that
\begin{gather} 
e^{-\lambda} \| f\eta \| \geq C_0 \exp(-2\Lambda \tau) \gg C\exp(-cA^2\tau)
\end{gather}

The main steps of the argument are as follows: First, we derive pointwise bounds for $f\eta$ using standard techniques. Then, using this bound and Lemma \ref{betterbaseptest} and sharper estimates recorded in Corollaries \ref{reffeqlocest}, \ref{refquadleqi} and \ref{refidifineq}, we show that a bound $\| f\eta \|\leq C\exp(-\beta \tau)$ results in a better bound $\| F \| \leq C \exp(-\frac{3}{2}\beta \tau)$ for the non-linearity in the equation for $f\eta$. Using the variation of constants formula, it will then follow that for any $k>m$, the $k$-th mode will have a decay rate strictly faster than $\exp(-\lambda_m \tau)$. Finally, we will derive the ODE for the coefficient of $m$-th mode and show that it must have the exact $\exp(-\lambda_m \tau)$ decay rate.

We have the following standard apriori estimate. We refer to \cite{velher} for a proof.
\begin{thm} \label{class}
Consider the initial value problem
\begin{gather*}
    f_\tau = (\mathcal{L} + \frac{l}{2})u + F(\tau, \sigma)\\
    f(0 , \sigma) = f_0
\end{gather*}
where $l$ is an integer and the operator $\mathcal{L}$ is defined as
$$ \mathcal{L} = \frac{\partial^2}{\partial \sigma^2} - \frac{\sigma}{2} \frac{\partial}{\partial \sigma}+1$$
Assume $F(\tau, \sigma) \in L^2_{loc}\big( (0, \infty); L^2_\rho \big)$ and $f_0 \in L^2_\rho$. Then the initial value problem has a unique solution in the weighted Sobolev space $H^1_\rho$, and we have the following estimates for the $L^2_\rho$ norms for $\tau\in(0,T)$:
\begin{gather*}
    \| f \|^2(\tau) + \tau \|f_\sigma \|^2(\tau) + \int_0^T \tau^2 \|f_{\sigma \sigma}\|^2(\tau) \, d\tau \leq C_T \left( \|f_0\|^2 + \int_0^T \| F\|^2(\tau) \, d\tau \right)
\end{gather*}
Here $T$ is arbitrary and $C_T$ is a constant depending only on $T$ and $l$.
\end{thm}

\begin{lem}
On any compact set, we have
\begin{gather*}
    |f \eta| \leq C \exp(- \beta \tau)
\end{gather*}
\end{lem}
\begin{proof}
Note that in the equation for $f\eta$, the right hand side $F+E$ satisfies
$$\| F+E \|\leq C (\| f\eta \|+ \| f^2\theta^2 \|) \leq C (\| f\eta \| + I) \leq C \exp(-\beta \tau)$$
So the Theorem applied on $[\tau-2, \tau]$ gives
\begin{align*}
    & 2\|(f\eta)_\sigma \|^2(\tau) \leq C_2 \left( \|f\eta\|^2(\tau-2) + \int_{\tau-2}^\tau \| F\|^2(s)+\|E\|^2(s) \, ds \right) \leq \\
    & C e^{- 2\beta (\tau-2)} + \frac{C}{2\beta} \big( e^{-2\beta(\tau-2)} - e^{-2\beta\tau}\big) \leq C \exp(- 2\beta \tau)
\end{align*}
By Sobolev inequality, we get the desired estimate.
\end{proof}

\begin{lem}
If $a$ in Proposition \ref{idifineq} and $A$ in our cut-offs are large compared to $\Lambda$, the bound $\| f\eta \| \leq C \exp(- \beta \tau)$ implies $I \leq C \exp(- 2\beta \tau)$ and $\| F \| + \| E \| \leq C \exp(- \frac{3}{2}\beta \tau)$ for large $\tau$.
\end{lem}
\begin{proof}
Consider the differential inequality for $I$ in Proposition \ref{idifineq}. In Corollary \ref{refidifineq}, we choose $\delta$ to be a small number. By the previous Lemma, we have $|f| \leq C \exp(- \beta \tau)$ for $|\sigma| \leq \delta^{-1}$. So $I$ satisfies
\begin{gather*}
    \frac{d}{d\tau} I \leq -a I + C e^{- \beta \tau} \| f\theta \| + C\exp(-cA^2\tau) \leq -aI + C e^{- \beta \tau} \| f \eta \| \leq -aI + C e^{- 2\beta \tau}
\end{gather*}
where we have used the smallness of $\exp(-cA^2\tau)$ compared to $Ce^{-\beta \tau}\| f\eta \|$. Now if $a \geq 2 \Lambda > 2\beta$, we get
\begin{gather*}
    I(\tau) \leq e^{a(\tau_0 - \tau)} I(\tau_0)+ C e^{-a\tau} \int_{\tau_0}^\tau e^{as} e^{-2\beta s}\, ds \leq e^{a(\tau_0 - \tau)} I(\tau_0) + C e^{-2\beta \tau} \leq C e^{-2\beta \tau}
\end{gather*}

Now, in Corollary \ref{refquadleqi}, we can pick $\delta$ to be a small number and use the pointwise bound $\varepsilon(\tau) = C \exp(- \beta \tau)$ for $|f|$ to get
\begin{gather*}
    \| f^2 \theta^2 \| \leq e^{-\beta \tau} \| f\eta\| + \delta I \leq e^{-2\beta \tau}
\end{gather*}
On the other hand, by Lemma \ref{betterbaseptest}, we have $\frac{1}{u^2( \tau,0)}-1 \leq C \exp(- \beta \tau)$. So, in Corollary \ref{reffeqlocest}, we can take $\delta(\tau)\leq C e^{- \beta \tau}$ and $\varepsilon(\tau)= \exp(- \frac{\beta}{2} \tau)$ and get
\begin{gather*}
    \| F \| \leq C e^{-\frac{\beta}{2}\tau} \| f\eta \| + C e^{\frac{\beta}{2}\tau} \| f^2 \theta^2 \| \leq C e^{-\frac{3}{2}\beta \tau}  
\end{gather*}
\end{proof}

We recall the variation of constants formula: Suppose $\mathcal{A}$ is an operator with eigenvalues $-\lambda_k$ and eigenfunctions $h_k(\sigma)$ for $k=0,1,2, \dots$, and assume $f(\tau, \sigma)$ is a solution to
\begin{gather*}
    \frac{\partial}{\partial \tau}f = \mathcal{A}f + F \qquad f(\tau_0, \sigma) = f_0(\sigma)= \sum_{k=0}^\infty a_k h_k
\end{gather*}
Then $f$ can be expressed as
\begin{gather}\label{varconst}
    f = \sum_{k=0}^\infty \left( e^{\lambda_k(\tau_0 - \tau)}a_k+\int_{\tau_0}^\tau e^{\lambda_k(s-\tau) }\langle F, h_k \rangle \, ds \right) h_k
\end{gather}

The following is inspired by \cite{velher}.
\begin{lem} \label{tail}
Assume $\| f\eta \| \leq C\exp(-\beta \tau)$ with $\lambda_{m-1} < \beta < \lambda_m$ and $m \geq 2$. We have the estimate
\begin{gather*}
    \| \pi_{\geq m+1} (f\eta) \| \leq C \exp \left( - \beta^* \tau \right)
\end{gather*}
where $\beta^* = \min\{ \frac{4}{3}\beta, \lambda_{m+1}\}$. Since $m \geq 2$, by choosing $\beta$ close to $\lambda_m$ and taking $\delta$ small enough, we have 
$$\lambda_m + \delta < \beta^*$$
\end{lem}
\begin{proof}
For $\tau \geq \tau_0$, $f\eta$ satisfies the equation
\begin{gather*}
    \frac{d}{d\tau}f\eta = \mathcal{A}(f\eta) + F +E\\
    \| F + E \| \leq C \exp(-\frac{3}{2}\beta \tau)
\end{gather*}
By the variation formula \ref{varconst}, $\pi_{\geq m+1}(f\eta)$ is equal to
\begin{gather*}
    \sum_{k=m+1}^\infty e^{\lambda_k(\tau_0 - \tau)} \langle f\eta(\tau_0), h_k \rangle h_k + \sum_{k=m+1}^\infty \left( \int_{\tau_0}^\tau e^{\lambda_k(s-\tau) }\langle F+E, h_k \rangle \, ds \right) h_k
\end{gather*}

For the first sum, since $k \geq m+1$, we have for $\tau>\tau_0+1$
\begin{align*}
    & \left\| \sum_{k=m+1}^\infty e^{\lambda_k(\tau_0 - \tau)} \langle f\eta(\tau_0), h_k \rangle h_k \right\| \leq \\
    & e^{-\frac{m}{2}(\tau - \tau_0)}\| f\eta(\tau_0) \|\sum_{k=0}^\infty e^{\frac{-k}{2}(\tau - \tau_0)} \leq C e^{-\frac{m}{2}\tau} e^{(\frac{m}{2}-\beta)\tau_0}
\end{align*}

For the second sum, we proceed as follows. For any term with $m+1 \leq k \leq 12m$, we have 
\begin{gather*}
    \left| \int_{\tau_0}^\tau e^{\lambda_k(s-\tau) }\langle F+E, h_k \rangle \, ds \right| \leq \int_{\tau_0}^\tau e^{\lambda_k(s-\tau) } \| F+E \|\, ds \leq Ce^{-\lambda_k\tau}\int_{\tau_0}^\tau e^{(\lambda_k-\frac{3}{2}\beta)s} \, ds
\end{gather*}
which is bounded by $C\exp(-\frac{3}{2}\beta\tau)$ if $\lambda_k > \frac{3}{2}\beta$ and by $C\exp(-\lambda_k\tau)$ if $\lambda_k < \frac{3}{2}\beta$ (We can assume $\frac{3}{2}\beta \neq \lambda_k$ for every $k$ by slightly changing $\beta$ if necessary).

Now consider the terms for $k \geq 12m+1$. We estimate the integral on $[\tau_0, 8\tau/9]$ and on $[8\tau/9, \tau]$ separately. On $[\tau_0, 8\tau/9]$ the norm is bounded by
\begin{align*}
    &\int_{\tau_0}^\frac{8\tau}{9} \left( \sum_{k=12m+1}^\infty e^{2\lambda_k(s-\tau)} \right)^\frac{1}{2} \|F + E\|(s) \,ds \leq \\
    & e^{-6m \frac{\tau}{9}}\int_{\tau_0}^\frac{8\tau}{9} C \exp(-\frac{3}{2}\beta s) \, ds \leq  C e^{-6m \frac{\tau}{9}} \exp(-\frac{3}{2}\beta \tau_0)
\end{align*}
where we have used $\sum_{k=12m+1}^\infty e^{-2\lambda_kr} \leq C e^{-12mr}$ for $r\geq 1$. On $[8\tau/9, \tau]$ we have $e^{\lambda_k(s-\tau)} \leq 1$ and the norm is bounded by
\begin{align*}
    & \int_{\frac{8\tau}{9}}^\tau \| F+E \| \, ds \leq \int_{\frac{8\tau}{9}}^\tau \exp(-\frac{3}{2}\beta s)\, ds \leq C \exp(-\frac{3}{2}\beta. \frac{8\tau}{9}) = C \exp(-\frac{4}{3}\beta \tau)
\end{align*}
\end{proof}

Now we derive the ODE for the coefficient of the $m$-th mode.
\begin{prop}\label{onemodedoms}
Let
\begin{gather}
    f\eta = \sum_{k=0}^\infty a_k(\tau) h_k(\sigma)
\end{gather}
be the eigenfunction expansion of $f\eta$ in $L^2(\rho\, d\sigma)$. Assume we have exponential convergence in Proposition \ref{MZ}. If the supremum of all exponential decay rates is $\lambda = \lambda_m$, then
$$a_m(\tau) = C\exp(-\lambda_m \tau)(1+o(1))$$
Furthermore, for any  $\varepsilon$, there exists $\tau_{\varepsilon}$ such that $\| \pi_{\neq m} (f\eta) \|^2 \leq \varepsilon a_m^2$ for $\tau \geq \tau_{\varepsilon}$.
\end{prop}
\begin{proof}
Choose $\beta$ close enough to $\lambda_m$ according to Lemma \ref{tail}. By the previous Lemmas, we have $I \leq C e^{-2\beta \tau}$ and $\| F+ E\| \leq C e^{-\frac{3}{2}\beta \tau}$. The ODE for the coefficient $a_m$ is
\begin{gather*}
    \frac{d}{d\tau}a_m = -\lambda_m a_m + \langle F+E \, , \, h_m\rangle
\end{gather*}

First, assume that for some $\varepsilon_0$ and some $\tau_*$, we have $\exp(-\frac{3}{2}\beta \tau) \geq \varepsilon_0 |a_m|$ for $\tau\geq \tau_*$. This would imply $|a_m| \leq C\exp(-\frac{3}{2}\beta \tau)$. But by Lemma \ref{tail}, we have $\| \pi_{\geq m+1} (f\eta)\|\leq C \exp(-(\lambda_m+\delta) \tau)$, and combining this with Proposition \ref{MZ2} we have
$$\| \pi_{\leq m-1} (f\eta)\|\leq C \left(\| \pi_{\geq m} (f\eta)\|+I + \exp(-cA^2\tau) \right) \leq C \exp(-(\lambda_m+\delta) \tau)$$
which gives a better bound for $\| f\eta\|$ and contradicts $\lambda_m$ being the supremum of decay rates. Therefore, for any $\varepsilon$ we have
\begin{gather*}
    \exp(-\frac{3}{2}\beta \tau_i) \leq \varepsilon |a_m(\tau_i)| \qquad \text{for} \,\, \tau_i \to \infty
\end{gather*}

Now let $\| F+E \|\leq C_0 \exp(-\frac{3}{2}\beta \tau)$, and choose any small $\varepsilon$ and $\tau_j$ such that $\lambda_m+\varepsilon < \frac{3}{2}\beta$ and $2C_0 \exp(-\frac{3}{2}\beta \tau_j) \leq \varepsilon |a_m(\tau_j)|$. We may assume $a_m(\tau_j)$ is positive infinitely often since the other case is similar. Then, as long as $C_0 \exp(-\frac{3}{2}\beta \tau) \leq \varepsilon |a_m(\tau)|$ holds, we would actually have
\begin{align*}
    &|\frac{d}{d\tau}a_m + \lambda_m a_m| \leq |\langle F+E \, , \, h_m\rangle| \leq \| F+E \| \leq C_0 \exp(-\frac{3}{2}\beta \tau) \leq \varepsilon |a_m|\\
    & \Rightarrow a_m(\tau) \geq a_m(\tau_j) e^{-(\lambda_m+\varepsilon) (\tau-\tau_j)} \geq \frac{2C_0}{\varepsilon} e^{-\frac{3}{2}\beta \tau_j} e^{-(\lambda_m+\varepsilon) (\tau-\tau_j)} > \frac{2C_0}{\varepsilon} e^{-\frac{3}{2}\beta \tau}
\end{align*}
Therefore $2C_0 \exp(-\frac{3}{2}\beta \tau) \leq \varepsilon |a_m(\tau)|$ for all $\tau \geq \tau_j$. Hence, the ODE for $a_m$ follows. The estimate on $\| \pi_{\neq m} (f\eta) \|$ follows from previous statements.
\end{proof}

The estimate $$\| \pi_{\neq m} (f\eta) \|^2 \leq \varepsilon a_m^2$$ is our Domination Condition \ref{onemodedoms}, so we can establish case $3$ in our main Theorem as well.

\section{Appendix: Proof of ODE lemma \ref{myMZ}}
We discuss the proof of our adaptation of Merle-Zaag lemma. The proof is a modification and a more careful analysis of the original argument in \cite{merlezaag}, so we only discuss the differences. Consider the system
\begin{gather}
    \frac{dx}{d\tau}  \geq \frac{1}{2}x - \varepsilon(x+y+z)- B\exp(-b\tau) \nonumber \\
    |\frac{dy}{d\tau} | \leq  \varepsilon(x+y+z) + B\exp(-b\tau)\\
    \frac{dz}{d\tau}  \leq -\frac{1}{2}z + \varepsilon(x+y+z)+B\exp(-b\tau) \nonumber
\end{gather}
for $\tau \geq \tau_\varepsilon$. We first show the following:

\textbf{Claim 1}: Either $x(\tau) \geq Ce^{\frac{1}{10}\tau}$, or for every $\varepsilon \leq \frac{1}{10}$ we have $x \leq 20B e^{-b\tau} + 4\varepsilon(y+z)$.

\textit{Proof}: Fix some small $\varepsilon< \frac{1}{10}$, and consider $\beta = x-4\varepsilon(y+z)$. Assume that at some $\tau_0 \geq \tau_\varepsilon$, we have $x(\tau_0) > 20B e^{-b\tau_0}$ and $\beta(\tau_0) \geq 0$. Similar to \cite{merlezaag}, we can show
$$x(\tau) > 10B e^{-b\tau} \, \, , \,\, \beta(\tau) \geq 0 \Rightarrow \beta'(\tau) >0$$
But as long as $\beta>0$, the inequality for $x$ gives
$$\frac{dx}{d\tau} \geq (\frac{1}{4}-\varepsilon)x -B\exp(-b\tau)$$ which implies $x \geq C e^{\frac{1}{8}\tau}$. So for our fixed $\varepsilon$, we can assume $x \leq 20B e^{-b\tau_0}$ or $x \leq 4\varepsilon(y+z)$ at any time. Since $\varepsilon$ was arbitrary, we have the claim.

By substituting $x \leq 20B e^{-b\tau} + 4\varepsilon(y+z)$ and taking $\varepsilon$ small, we get
\begin{gather}
    |\frac{dy}{d\tau} | \leq  2\varepsilon(y+z) + 2B\exp(-b\tau)\\
    \frac{d z}{d\tau}  \leq -\frac{1}{2}z + 2\varepsilon(y+z)+2B\exp(-b\tau) \nonumber
\end{gather}

\textbf{Claim 2}: Fix an $\varepsilon<\frac{1}{100}$ and a number $\alpha>10$ such that $\alpha \varepsilon < \frac{1}{100}$. Consider $\gamma = \alpha \varepsilon y - z - \frac{10B}{b}\exp(-b\tau)$. If $\gamma(\tau_1)>0$ for some $\tau_1 \geq \tau_\varepsilon$, then $\gamma(\tau) \geq 0$ for all $\tau \geq \tau_1$, and $c_\varepsilon e^{-4\varepsilon\tau} \leq y(\tau) \leq C_\varepsilon e^{4\varepsilon \tau}$.

\textit{Proof}: Similar to \cite{merlezaag}, we can show that at any point $\tau \geq \tau_1$ we have
$$\gamma(\tau) = 0 \Rightarrow \gamma'(\tau)>0$$

In \cite{merlezaag}, the fixed coefficient $\alpha=8$ is considered, but we need and arbitrary number later in the Appendix. Note that the $\frac{10B}{b}\exp(-b\tau)$ term is $\gamma$ and the bound on $\alpha \varepsilon$ help us absorb the exponential terms after differentiation. Substituting $z + \frac{10B}{b}\exp(-b\tau) \leq \alpha \varepsilon y $ in the inequalities before Claim 2 and using $\alpha \varepsilon<C$ and $b\gg 1 \gg \varepsilon$ we get
\begin{gather}
    |\frac{dy}{d\tau}| \leq  4\varepsilon y + 4B\exp(-b\tau) \Rightarrow c_\varepsilon e^{-4\varepsilon\tau} \leq y(\tau) \leq C_\varepsilon e^{4\varepsilon \tau}
\end{gather}

\textbf{Claim 3}: Fix an $\varepsilon<\frac{1}{100}$ and a number $\alpha>10$ such that $\alpha \varepsilon < \frac{1}{100}$, and consider $\gamma$ as in Claim 2. If $\gamma \leq 0$ for all $\tau \geq \tau_\varepsilon$, then
\begin{gather}
    x+y \leq C(1+\frac{1}{\varepsilon}) \left( z+ B\exp(-b\tau) \right)\\
    \frac{dz}{d\tau} \leq -(\frac{1}{2} - 2\varepsilon -\frac{2}{\alpha})z + 4B\exp(-b\tau) \nonumber
\end{gather}
and hence $z \leq C_\varepsilon \exp\left( -(\frac{1}{2} - 2\varepsilon -\frac{2}{\alpha})\tau \right)$

To finish the proof of the ODE lemma, note that if Claim 2 holds for \textit{some} small $\varepsilon$, then Claim 3 cannot hold for \textit{any} small $\varepsilon'$, because one implies slow decay of $y$ with rate $4\varepsilon$ at most, and the other implies fast decay of a rate at least $\frac{1}{2} - 2\varepsilon -\frac{2}{\alpha}$. So exactly one of the claims holds for \textit{every} $\varepsilon$. In Claim 3, by choosing $\alpha$ large and $\varepsilon$ small we can get as close to $\frac{1}{2}$ as we want.

\nocite{*}

\bibliographystyle{amsalpha}
\bibliography{refs}

\end{document}